\documentclass[10pt]{amsart}
\usepackage[english]{babel}
\usepackage{amsmath}
\usepackage{amssymb}
\usepackage{amsthm}
\usepackage{libertine}
\usepackage[all]{xy}
\usepackage{color}
\usepackage{graphicx}
\usepackage[hidelinks]{hyperref}
\usepackage{cleveref}
\usepackage{fullpage}
\usepackage{comment}
\usepackage{tikz-cd}
\usepackage{tikz}

 \newtheorem{thm}{Theorem}[section]
 \newtheorem{cor}[thm]{Corollary}
 \newtheorem{lem}[thm]{Lemma}
 \newtheorem{prop}[thm]{Proposition}
  
 \theoremstyle{definition}
 \newtheorem{defn}[thm]{Definition}
 \theoremstyle{remark}
 \newtheorem{rem}[thm]{Remark}

\newtheorem{ex}{Example}
\numberwithin{equation}{section}
\usepackage{amscd}

\newcommand{\scal}[1]{\left<#1\right>}

\newcommand{\R}{\mathbb{R}}      

\newcommand{\C}{\mathbb{C}}

\title[The Fueter-Sce mapping and the Clifford-Appell polynomials ]{The Fueter-Sce mapping and the Clifford-Appell polynomials }

\begin{document}
\date{}
\author{Antonino De Martino, \, Kamal Diki, \, Ali Guzmán Adán }
\maketitle
\begin{abstract}
The Fueter-Sce theorem provides a procedure to obtain axially monogenic functions, which are in the kernel of generalized Cauchy-Riemann operator in $ \mathbb{R}^{n+1}$. This result is obtained by using two operators. The first one is the slice operator, which extends holomorphic functions of one complex variable to slice monogenic functions in $ \mathbb{R}^{n+1}$. The second one is a suitable power of the Laplace operator in $n+1$ variables. Another way to get axially monogenic functions is the generalized Cauchy-Kovalevskaya (CK) extension. This characterizes axial monogenic functions by their restriction to the real line. In this paper, using the connection between the Fueter-Sce map and the generalized CK-extension, we explicitly compute the actions  $\Delta_{\mathbb{R}^{n+1}}^{\frac{n-1}{2}} x^k$, where $x \in \mathbb{R}^{n+1}$. The expressions obtained is related to a well-known class of Clifford-Appell polynomials. These are the building blocks to write a Taylor series for axially monogenic functions. Moreover, we focus on some elementary axially monogenic functions, where the action of the Fueter-Sce map and the generalized CK-extension coincide. In order to get algebraic relations between the elementary functions, as in complex analysis, we define a new product between axially monogenic functions. By using the connections between the Fueter-Sce map and the generalized CK extension  we characterize the range and the kernel of the Fueter-Sce map. Furthermore, we focus on studying the Clifford-Appell-Fock space and the Clifford-Appell-Hardy space. Finally, using the polyanalytic Fueter-Sce theorems we obtain a new family of polyanalytic monogenic polynomials, which extends to higher dimensions the Clifford-Appell polynomials.
\end{abstract}


\tableofcontents

\section{Introduction}
The Fueter-Sce theorem (see \cite{CSS2,S}) and the generalized Cauchy-Kovalevskaya (CK) extension (see \cite{green}) are the main tools of Clifford analysis to transform analytic functions of one (real or complex) variable into monogenic functions. These functions are null-solutions of the generalized Cauchy-Riemann operator in $ \mathbb{R}^{n+1}$. The so-called Fueter-Sce map is the differential operator $ \Delta^{\frac{n-1}{2}}_{\mathbb{R}^{n+1}}$, where $n$ is a fixed odd number and $\Delta_{\mathbb{R}^{n+1}}$ is the Laplacian in $n+1$ variables. 
While the generalized CK-extension is written as a suitable power series of  $\underline{x} \partial_{x_0}$ (see Theorem \ref{GCK}), where $ \underline{x}$ is the vector variable in $ \mathbb{R}^{n}$ and $x_0$ is the real variable.

Both maps have been proved to be related but are not equal. Indeed, the generalized CK extension is an isomorphism, while the Fueter-Sce map is a pointwise differential operator of $n+1$ variables, and therefore it is not injective. However, in some cases, these two extension tools yield identical actions (up to a multiplicative constant). For instance, these are the cases of the exponential, the trigonometric and hyperbolic functions. This is one of the reason why, the Fueter-Sce map has successfully been used in the literature to extend these algebraic functions and some related functional spaces to $ \mathbb{R}^{n+1}$.  Our first goal is thus to study the relations under which both actions coincide.

\medskip

The connection between the generalized CK extension and the Fueter-Sce map is crucial to figure out the behaviour of the Fueter-Sce map applied to the basic element $x^k$, where $x=x_0+ \underline{x} \in \mathbb{R}^{n+1}$ is a Clifford paravector (see Section 3) and $k \in \mathbb{N}_0$. Indeed, as far as the authors know, it is hard to explicitly compute the action $ \Delta^{\frac{n-1}{2}}_{n+1} x^k$ via direct computations. Instead, this action is easily computed with the help of the generalized CK extension (see \cite{DDG}). Such an action have been computed to be
\begin{equation}
\label{fund0}
\Delta^{\frac{n-1}{2}}_{\mathbb{R}^{n+1}}(x^{k})=\frac{(-1)^{\frac{n-1}{2}}2^{n-1}}{(n-1)!} \left[\Gamma\left(\frac{n+1}{2}\right)\right]^2\frac{k!}{(k-n+1)!}P^n_{k+1-n}(x), \qquad k> n-1
\end{equation}
where $P_{k+1-n}^n(x)$ are the so called Clifford-Appell polynomials, and are defined on Section 2. 
\\We note that polynomials $P_{k+1-n}^n(x)$ satisfy the so called Appell property. In classical terms a sequence of polynomials $ \{A_k\}_{k \in \mathbb{N}}$ of the real variable $x_0$ is called Appell sequence if 
$$ \frac{d}{d x_0} A_k=k A_{k-1}.$$
These polynomials were first introduced by the French mathematician Paul Emile Appell in  \cite{Appell1880}. The well-known Hermite, Bernoulli and Euler polynomials form an interesting family of examples of such polynomials.
Appell polynomials are important in several areas of mathematics and their applications. For example they are relevant in probability theory and stochastic process since they can be related to random variables, see \cite{Bao2015}, they were used also to study optimal stopping problems related to Lévy process in \cite{Salm2011}. Also in  Clifford analysis Appell polynomials were investigated and studied by several authors with respect to the hypercomplex derivative, see \cite{CFM,CMF, MF2007,Pena2011}. Moreover, they were used to study a Bargmann-Fueter transform, see \cite{DKS}. 
\\ Formula \eqref{fund0} is essential in the whole paper and it is crucial to address the following problems.
\newline
\newline
\textbf{Problem}
Is it possible to expand Taylor series the axially monogenic functions in terms of the Clifford-Appell polynomials?
\newline
\newline
In Clifford analysis the Taylor series expansion in terms of the Fueter polynomials is well-known. Fueter polynomials describe all monogenic functions as they play the same role in Clifford analysis as the monomials $ x_0^{\alpha_0}x_{1}^{\alpha_1}x_{2}^{\alpha_2}...x_{n}^{\alpha_n}$
in several real variables analysis. However there is a simpler way to describe monogenic functions defined on axially symmetric domains. This description has axially monogenic functions as central objects, see \cite{Somm1}. By restricting our analysis to axially monogenic functions, we replace the more complicated approach of Fueter polynomials, by the simpler and more accessible approach of Clifford Appell polynomials which play the role of $x_0^k$ in one-variable real analysis. 


\medskip
We will study elementary functions like the exponential, trigonometric and hyperbolic functions in the monogenic setting in terms of the Clifford-Appell polynomials. 
For this set of functions the actions of the Fueter-Sce map and generalised CK extension coincides. The methodology is to apply the Fueter-Sce map (or the generalized Cauchy-Kovalevskaya extension) to a slice monogenic function. A similar approach is done in \cite{G}. Even in that case the authors study the effects of the Fueter-Sce map acting on a slice monogenic function. However, they do not use the expansion in series for the slice monogenic functions and so they did not get a full description of elementary functions in terms of spherical-monogenic polynomials. 
We overcome this issue by applying the Fueter-Sce map to a slice monogenic function expanded in series and by using formula \eqref{fund0}. Therefore, we get expansions series written in terms of the Clifford-Appell polynomials.
\\ Another interesting problem concerning the elementary functions is the following one.
\newline
\newline
\textbf{Problem}
Is it possible to extend the algebraic identities that hold in complex analysis to the trigonometric and hyperbolic functions in Clifford analysis?
\newline
\newline
The key point to solve this problem is to work with a suitable product. In literature, the only product that preserves the monogenicity is the so called CK-product, see \cite{green}. This product is related to the classical CK-extension, whose main idea is to give a catheterization of solutions of suitable systems of PDE's by their restriction to a submanifold of codimension one. If the PDE involved is the Cauchy-Riemann equations in $ \mathbb{R}^{n+1}$, we can characterize a monogenic function by its restriction to the hyperplane $x_0=0$.
The CK-product, denoted by $\odot_{CK}$, associates to any pair of monogenic functions $f$, $g$  the CK extension of the punctual product between $f$ and $g$ restricted to $\mathbb{R}^n$, namely
$$ f \odot_{CK} g=CK[ f_{|\mathbb{R}^n} \cdot g_{| \mathbb{R}^n}].$$
 
In this paper, instead of working with the CK-product, we introduce a new product, called generalized CK-product and denoted by $\odot_{GCK}$. In this specific case we work in codimension $n$, i.e. we characterize axially monogenic functions in $\mathbb{R}^{n+1}$ by their restrictions to the real line. This product, compared to the CK product (see \eqref{Ck}), gives a more natural multiplication rule for the product of Clifford-Appell polynomials:
$$ (P_k^n \odot_{GCK} P_s^n)(x)=P_{k+s}^n(x), \qquad \forall k,s \geq 0.$$
\\ We focus on solving the following problem.
\newline
\newline
\textbf{Problem}
Study the range and the kernel of the Fueter-Sce map.
\newline
\newline
In order to solve the this problem we develop a general apporach that can be adapted for specific modules of monogenic functions. In this paper we give some examples related to the Hardy, Bergman, Dirichlet and Fock spaces in the axially monogenic setting. Moreover, by means of the generalized CK-extension we define different axially monogenic Fock and Hardy spaces. We prove that it is possible to write their reproducing kernels in terms of the Clifford-Appell polynomials. Using this Appell system we introduce the notions of creation and annihilation operators for the Fock space and we study different properties of the shift and backward operators in the Hardy space. 
Our last task in this setting is solving the following problem.
\newline
\newline
\textbf{Problem} Is it possible to extend the Clifford-Appell polynomials in the setting of polyanalytic functions?
\newline
\newline
In complex analysis, the theory of polyanalytic functions is an extension of the classical theory of holomorphic functions. Its main goal is to study functions in the kernel of a certain power of the Cauchy-Riemann operator. In the monogenic  setting, a theory of polyanalytic functions of any order is presented in \cite{B1976}. However, there is no full description of a basis of the module of polyanalytic functions. In this work, we provide a series expansion of a polyanalytic function of order $m+1$ in terms of suitable homogenous polynomials.

\medskip
Summarizing, in this paper we will solve the following problems

\begin{itemize}
\item find a Taylor series expansion in terms of the Clifford-Appell polynomials for axially monogenic functions,
\item extend the classical identities of complex analysis, where trigonometric and hyperbolic functions are involved, to the Clifford setting,
\item characterize the range and the kernel of the Fueter-Sce map. 
\item generalize the Clifford-Appell polynomials to the polyanalytic setting.
\end{itemize}

\emph{Outline of the paper}: Besides this introduction the paper consists of 8 sections. In Section 2, we recall some results on Clifford analysis such as the Fueter-Sce mapping theorem, the generalized CK theorem and the connection between them. In Section 3, we study the action of the Fueter-Sce map on the Clifford paravector monomials $x^k$ and relate them to the Clifford-Appell polynomials. We also investigate the hypothesis under which the Fueter-Sce map and the generalized CK-extension coincide, up to a constant. The conditions hold for the elementary functions in the monogenic setting. In Section 4, we determine the range and kernel of the Fueter-Sce map and study how the Fueter-Sce map defines an isometry between suitable subspaces of monogenic and slice monogenic functions. Section 5 is devoted to the Clifford-Fock module. We study the creation and annihilation operators in this setting. In Section 6, we focus on studying the Clifford-Hardy module and we investigate the properties of the shift and backward operators in the Clifford setting. Finally, in Section 7, we introduce a generalized family of Clifford-Appell polynomials in the polyanalytic monogenic setting. This construction is based on the polyanalytic monogenic Fueter-Sce mapping theorem, see \cite{ADS}.

\section{Preliminary results}
The real Clifford algebra $ \mathbb{R}_n$ is generated by the standard basis $ \{e_1,...,e_n\}$ of $ \mathbb{R}^n$, whose elements satisfy the multiplicative relations $ e_{\ell}e_m+e_{m}e_{\ell}=0$ if $ \ell \neq m$ and $e^{2}_{\ell}=-1$, otherwise. Any element in the Clifford algebra can be written as $$ \sum_{A} e_A x_A,$$ where for any multi-index $ A= \{\ell_1,..., \ell_r\} \subset \{1,2,...,n\}$, with $ \ell_1 <...<\ell_{r}$, we put $e_{A}=e_{\ell_1}e_{\ell_2}...e_{\ell_r}$, $e_{\emptyset}=1$. Any arbitrary element $x \in \mathbb{R}_n$ can be written in the form 
$$  x= \sum_{k=0}^n [x]_k, \qquad \hbox{where} \quad [x]_{k}= \sum_{|A|=k} x_A e_A,$$
where $[.]_k: \mathbb{R}_n \to \mathbb{R}^{(k)}_n$ is the canonical projection of $ \mathbb{R}_n$ onto the space of $k$-vectors $\mathbb{R}^{(k)}_n:=\hbox{span} \{e_A \, : \, |A|=k\}$. We define the real part of a Clifford number $x$ as $ \hbox{Re}(x)=[x]_0$.
We denote the norm of an element in the Clifford algebra as $|x|:= \left(\sum_{A} x_A^2\right)^{\frac{1}{2}}$. We recall that for every $x$, $y \in \mathbb{R}_n$ we have $|xy| \leq 2^{\frac{n}{2}} |x| |y|$, see \cite{BDS}.

An element $(x_0,x_1,...,x_n) \in \mathbb{R}^{n+1}$ will be identified with the Clifford element $x=x_{0}+ \underline{x}=x_{0}+ \sum_{\ell=1}^n x_{\ell} e_{\ell}$. Elements of this form are called paravectors. The norm of $x \in \mathbb{R}^{n+1}$ is defined as $|x|^2=x_{0}^2+x_{1}^2+...+x_{n}^2$. The conjugate of $x$ is defined by $\bar{x}= x_0- \underline{x}=x_{0}- \sum_{\ell=1}^n x_{\ell} e_{\ell}$. 
We shall denote by $ \mathbb{S}^{n-1}$ the $(n-1)$-dimensional sphere in $ \mathbb{R}^{n+1}$, i.e.
$$ \mathbb{S}^{n-1}= \{\underline{x}=e_1x_1+...+e_nx_{n} \, | \, x_{1}^2+...+x_n^2=1\}.$$
For $I \in \mathbb{S}^{n-1}$ we have $I^{2}=-1$.
The vector space $\mathbb{C}_I:=\{u+Iv \,: \, u,v \in \mathbb{R}\},$ passing through 1 and $I \in \mathbb{S}^{n-1}$, is isomorphic to the complex plane. 
\begin{defn}
\label{slicemono1}
Let $\Omega \subset \mathbb{R}^{n+1}$ be an open set. A function $f: \Omega \to \mathbb{R}_n$ is said to be (left) slice monogenic if, for every $I \in \mathbb{S}^{n-1}$, the restrictions of $f$ to the complex plane $\mathbb{C}_I$ are holomorphic on $ \Omega \cap \mathbb{C}_I$ i.e.
$$ \left(\partial_u+I \partial_v\right)f(u+Iv)=0.$$
\end{defn}
We denote the set of left slice monogenic function on $U$ by $ \mathcal{SM}(U)$.
For for more details about the slice monogenic functions see \cite{CKS, CGS2013, CSS3, CSS4, RW}.
\\ Slice monogenic functions have good properties when defined in the following sets.
\begin{defn}
Let $U \subseteq \mathbb{R}^{n+1}$ be a domain. We say that $U$ is a \emph{slice domain (or s-domain)} if $U \cap \mathbb{R}$ is non empty and if $ \mathbb{C}_I \cap U$ is a domain in $ \mathbb{C}_I$ for all $I \in \mathbb{S}^{n-1}.$
\end{defn}

\begin{defn}
Let $U \subseteq \mathbb{R}^{n+1}$ be an open set. We say that $U$ is an \emph{axially symmetric domain} if for every $x= \hbox{Re}(x)+I | \underline{x}| \in U$ the whole sphere $\hbox{Re}(x)+|\underline{x}| \mathbb{S}^{n-1}:= \{\hbox{Re}(x)+I | \underline{x}|; \, I \in \mathbb{S}^{n-1}\}$ is contained in $U$.
\end{defn}

We can also give the definition of axially symmetric slice domain in the following way. 
\\ A set $ \Omega \subset \mathbb{R}^{n+1}$ is an axially symmetric slice domain if and only if there exist a complex intrinsic domain $D \subset \mathbb{C}$ such that
$$ \Omega=D \times \mathbb{S}^{n-1}= \{u+vI: \, (u,v) \in D\, , I \in \mathbb{S}^{n-1}\}.$$
The slice monogenic functions defined on axially symmetric slice domain enjoy good properties as stated in the following result.

\begin{thm}[Representation formula \cite{CSS3}]
\label{moons}
Let $U \subseteq \mathbb{R}^{n+1}$ be an axially slice symmetric open set and $ \mathcal{U} = \{(u,v) \in \mathbb{R}^2 \,; \, u+ \mathbb{S}^{n-1}v \subset U\}$. A function $f:U  \to \mathbb{R}_n$ is (left) slice monogenic if it is of the form
$$ f(x)= \alpha(u,v)+I \beta(u,v), \qquad x=u+Iv \in U,$$
where $ \alpha$, $ \beta: \mathcal{U} \to \mathbb{R}_n$ are differentiable functions that satisfy the "even-odd" conditions
\begin{equation}
\label{evenodd}
\alpha(u,v)=\alpha(u,-v), \qquad \beta(u,v)=- \beta(u,-v), \qquad \hbox{for all} \quad (u,v) \in \mathcal{U}.
\end{equation}
and furthermore satisfy the Cauchy-Riemann equations
$$ \begin{cases}
\partial_u \alpha- \partial_v \beta=0,\\
\partial_v \alpha+ \partial_u \beta=0.
\end{cases}
$$
\end{thm}

\begin{rem}
If we assume that $U \cap \mathbb{R} \neq \emptyset$ than the function $f$ is well defined at the real points since $ \beta(x_0,0)=0$ for all $x=x_0 \in \mathbb{R}$.
\end{rem}

One of the main properties of the slice monogenic functions is that we can develop them in power series, see \cite{CSS4}.

\begin{thm}
\label{expa}
Let $f$ be a slice monogenic function on a s-domain $U \subseteq \mathbb{R}^{n+1}$. Then if $0 \in U$, the function can be represented in power series
$$ f(x)= \sum_{n=0}^\infty x^n a_{n},$$
on the ball $B(0,R)$, where $R$ is the largest positive real number such that $B(0,R)$ is contained in $U$.
\end{thm}
Now, we recall how to induce slice monogenic functions from the so-called holomorphic intrinsic functions.

\begin{defn}[Holomorphic intrinsic functions]
A holomorphic function $f(z)=\alpha(u,v)+i \beta(u,v)$ is said to be \emph{intrinsic} if it is defined in an intrinsic complex domain $D$ and satisfies that $f(z)^c=f(z^c)$, where $.^c$ denotes the complex conjugation. This means that real part and imaginary part of $f$ satisfy the condition \eqref{evenodd}, or equivalently that the restriction $f|_{\mathbb{R}}$ of $f$ to the real line is $ \mathbb{R}$-valued. 
\end{defn}
We denote by $\mathcal{O}(D)$ the space of holomorphic complex functions on $D$ and by $ \mathcal{H}(D)$ the (real) vector sub-space of $ \mathcal{O}(D)$ complex holomorphic intrinsic functions i.e.
\begin{eqnarray*}
\mathcal{H}(D)&=& \{f \in \mathcal{O}(D)\, : \, \alpha(u,v)=\alpha(u,-v), \, \beta(u,v)=-\beta(u,-v)\}\\
&=& \{f \in \mathcal{O}(D)\, : \, f|_{\mathbb{R}} \quad \hbox{is real valued}\}.
\end{eqnarray*}
The intrinsic holomoprhic functions defined in $D \subset \mathbb{C}$ induce slice monogenic functions on $\Omega= D \times \mathbb{S}^{n-1}$ by means of the extension map 
$$ S_{\mathbb{C}}: \mathcal{H}(D) \otimes \mathbb{R}_n \to \mathcal{SM}(\Omega), \qquad \alpha(u,v)+i \beta(u,v) \mapsto \alpha(x_0,|\underline{x}|)+I \beta(x_0,| \underline{x}|),$$
which consists of replacing the complex variables $z=u+iv$ by the paravector variable $x=x_0+ \underline{x}$, where the complex unit $i$ is replaced by the unit vector $I:= \frac{\underline{x}}{| \underline{x}|}$ in $\mathbb{R}^n$.

\medskip

It is also possible to extend real-analytic functions to slice monogenic functions. Let us consider $\Omega_1:= D \cap \mathbb{R}$. We denote by $ \mathcal{A}(\Omega_1)$ the space of rela-valued analytic functions defined on $\Omega_1$ with unique holomorphic extensions to $ D$. Then, we define the holomorphic extensions map $C:= \hbox{exp}(iv\partial_u)$, i.e.
$$ C: \mathcal{A}(\Omega_1) \to \mathcal{H}(D) \qquad f_{0}(u) \mapsto \sum_{j=0}^\infty \frac{(iv)^j}{j!} f^{(j)}_{0}(u).$$
By means of this we can define the slice monogenic extension map as $S= S_{\mathbb{C}} \circ C= \hbox{exp}(\underline{x} \partial_{x_0})$ i.e.
\begin{equation}
\label{slice1}
S: \mathcal{A}(\Omega_1) \otimes \mathbb{R}_n \to \mathcal{SM}(\Omega), \qquad f_0(x_0) \mapsto \sum_{j=0}^\infty \frac{\underline{x}^j}{j!}f^{(j)}_{0}(x_0).
\end{equation}
Now, we can state the following extension result.
\begin{thm}
Under the conditions considered above, we have the following isomorphic isomorphisms
$$ \mathcal{SM}(\Omega) \simeq \mathcal{A}(\Omega_1) \otimes \mathbb{R}_n \simeq \mathcal{H}(D) \otimes \mathbb{R}_n.$$
We have the following commutative diagram
\[
\begin{tikzcd}[row sep = 2em, column sep = 5em]
	\mathcal{A}(\Omega_1)\otimes \R_m \arrow[r, "C", ] \arrow[dr, "S", labels=below] & \mathcal{H}(D) \otimes \R_m \arrow[d, "S_\C"] \\
	& \mathcal{SM}(\Omega)
\end{tikzcd}
\]
where the map $S=S_\C\circ C =\exp(\underline{x}\partial_{x_0})$, given by $S[f_0](x) = \sum_{j=0}^\infty \frac{\underline{x}^j}{j!} \, f_0^{(j)}(x_0)$,
is inverted by the restriction operator to the real line, i.e.\ $S[f_0]\big|_{\underline{x}=0}=f_0$. 
\end{thm}
Another important extension of the holomorphic function theory to higher dimensions is given by the following.

\begin{defn}[Monogenic function]
\label{mono1}
Let $U \subset \mathbb{R}^{n+1}$ be an open set and let $f:U \to  \mathbb{R}_n$ be a function of class $C^1$. We say that $f$ is (left) monogenic on $U$ if
$$ \partial f(x):= (\partial_{x_0}+ \partial_{\underline{x}})f(x)= \left(\partial_{x_0}+ \sum_{i=1}^{n} e_i \partial_{x_i}\right)f(x)=0 \qquad \forall x \in U.$$
We denote by $ \mathcal{M}(U)$ the set of left monogenic functions on $U$.
\end{defn}
We observe that for $n=1$, monogenic functions on $ \mathbb{R}^2$ correspond to holomorphic function of the variable $x_0+e_1x_1$.
For more information about the notion of monogenic functions, see \cite{BDS, green, G}.
\\ A classic way to characterize monogenic functions  is the CK-extension, see \cite{BDS}. This is achieved by considering the restriction to the hyperplane $x_0=0$ of a monogenic function $f$. Precisely, the CK-extension of $f( \underline{x})$  is given by
$$ CK[f(\underline{x})](x)= \sum_{j=0}^\infty \frac{(-1)^j}{j!} x_{0}^j \partial_{\underline{x}}^j[f(\underline{x})].$$

A link between the set of slice monogenic functions and monogenic functions is the notion of axially monogenic function.
\begin{defn}[Axially monogenic functions]
Let $U$ be an axially symmetric slice domain in $\mathbb{R}^{n+1}$. A function $f: U \to \mathbb{R}_n$ is said to be axially monogenic, if it is monogenic and it is of the form
$$ f(x_0+ \underline{x})= A(x_0,|\underline{x}|) + \underline{\omega}B(x_0,|\underline{x}|), \qquad \underline{\omega}= \frac{\underline{x}}{|\underline{x}|},$$
where $A$ and $B$ are Clifford-valued functions that satisfy \eqref{evenodd}. We denote by $ \mathcal{AM}(U)$ the right $ \mathbb{R}_n$-module of axially monogenic functions on $U$.
\end{defn}
In literature, as for the case of slice monogenic functions, there exist an isomorphism between the modules of real-analytic Clifford-valued functions on the real line and the module of axially monogenic functions in $ \mathbb{R}^{n+1}$. This is a particular result of the generalized CK-extension. In this case we characterize axially monogenic functions in $\mathbb{R}^{n+1}$ by their restrictions to the real line.
\begin{thm}[Generalized CK-extension \cite{green}]
\label{GCK}
Let $\Omega_1 \subset \mathbb{R}$ be a real domain and consider $f_{0}(x_0)$ be a Clifford-valued analytic function in $\Omega_1$ . Then, there exists a unique sequence $ \{f_j( x_0)\}_{j=1}^\infty \subset \mathcal{A}(\Omega_1) \otimes \mathbb{R}_n$ such that the series
$$ f(x_0,\underline{x})= \sum_{j=0}^\infty \underline{x}^j f_{j}(x_0),$$
converges in an axially symmetric slice $(n+1)$-dimensional neighbourhood $ \Omega \subset \mathbb{R}^{n+1}$ of  $\Omega_1$ and its sum is a monogenic function i.e., $(\partial_{x_0}+\partial_{\underline{x}})f(x_0, \underline{x})=0$.
\\Furthermore, the sum $f$ is formally given by the expression
\begin{equation}
\label{Exx}
f(x_0, \underline{x})= \Gamma \left(\frac{n}{2}\right) \left( \frac{|\underline{x}|\partial_{x_0}}{2} \right)^{- \frac{n}{2}} \left(\frac{|\underline{x}|\partial_{x_0}}{2} J_{\frac{n}{2}-1}\left( |\underline{x}|\partial_{x_0} \right)+ \frac{\underline{x} \partial_{x_0}}{2} J_{\frac{n}{2}}\left(|\underline{x}|\partial_{x_0} \right) \right) f_{0}(x_0),
\end{equation}
where $J_{\nu}$ is the Bessel function of the first kind of order $\nu$. Formula \ref{Exx} is known as the generalized CK-extension of $f_0$, and it is denoted by $GCK[f_0](x_{0}, \underline{x})$. This extension operator defines an isomorphism between the right modules
$$ GCK: \mathcal{A}(\Omega_1) \otimes \mathbb{R}_n \to \mathcal{AM}(\Omega),$$
whose inverse is given by the restriction operator to the real line i.e. $GCK[f_0](x_0,0)=f_{0}(x_0)$.
\end{thm}
For a more general and deep treatment of this result we refer the reader to \cite{SS, SS1, G1}.

A relation between holomorphic intrinsic functions and axially monogenic functions is given by the following theorem, see \cite{CSS2, S}.
\begin{thm}[Fueter-Sce mapping theorem] Let $D \subseteq \mathbb{C}$ be an intrinsic complex domain and $n$ be a fixed odd number. Let $f_0(z)=f_0(u+iv)= \alpha(u,v)+i \beta(u,v)$ be an intrinsic holomorphic function defined in $D$. Then
$$ \Delta_{\mathbb{R}^{n+1}}^{\frac{n-1}{2}}[f(x_0+\underline{x})]=\Delta_{\mathbb{R}^{n+1}}^{\frac{n-1}{2}}[\alpha(x_0,r)+ \underline{\omega}\beta(x_0,r)],$$
is axially monogenic in axially symmetric slice domain $\Omega=D \times \mathbb{S}^{n-1}=\{(x_0, \underline{x}) \in \mathbb{R}^{n+1} \, : \, (x_0,|\underline{x}| \in D)\}.$
\end{thm}
The fact that every intrinsic holomorphic function $f$ is the unique holomorphic extension of a real analytic function $f_0$ on the real line, we can rewrite the Fueter-Sce theorem in the following way.
\begin{thm}[Fueter-Sce mapping theorem]
Let $\Omega_1 \subset \mathbb{R}$ be a real domain and $f_0 \in \mathcal{A}_0(\Omega_1) \otimes \mathbb{R}_n$. Then $ \Delta_{\mathbb{R}^{n+1}}^{\frac{n-1}{2}} \circ S[f_0](x_0, \underline{x})$ is an axial monogenic function on a $(n+1)$-dimensional axially symmetric slice neighbourhood $\Omega \subset \mathbb{R}^{n+1}$ of $\Omega_1$. 
\end{thm}

\begin{rem}
In \cite{Q} T.Qian proved that the Fueter-Sce theorem holds in the case of a Clifford algebra over an even number $n$ of imaginary units. In order to deal with the fractional laplacian he uses the techniques of Fourier multipliers in the space of distributions. We refer to \cite{CSS1,E1, E3, KQS, PSS1, PQS, S1} for several generalizations.
\end{rem}

As it is expected, both results the Fueter and the generalized CK extension theorems are related, see \cite[Thm. 4.2]{DDG}.
\begin{thm}
\label{FG}
Let $ D \subset \mathbb{C}$ be an intrinsic complex domain. Consider a holomorphic function $f: D \to \mathbb{C}$ such that its restriction to the real line is real valued. Then for $n$ odd we have
\begin{equation}
\label{CGK5}
\Delta^{\frac{n-1}{2}}_{\mathbb{R}^{n+1}}f(x_0+ \underline{x})=  \gamma_n GCK[f^{(n-1)}(x_{0})],
\end{equation}
where $\gamma_n:= \frac{(-1)^{\frac{n-1}{2}}2^{n-1}}{(n-1)!} \left[\Gamma\left(\frac{n+1}{2}\right)\right]^2.$ Setting $\Omega_1:=D \cap \mathbb{R}$ we have the following commutative diagram
\begin{equation}\label{Diag2}
\begin{tikzcd}[row sep = 3em, column sep = 6em]
\mathcal{A}(\Omega_1) \otimes \mathbb{R}_n \arrow[r, "S", rightarrow] \arrow[d, "\gamma_n \partial_{x_0}^{n-1}", labels=left] & \mathcal{SM}(\Omega) \arrow[d, "\Delta^{\frac{n-1}{2}}_{\mathbb{R}^{n+1}}"] \\
\mathcal{A}(\Omega_1) \otimes \mathbb{R}_n \arrow[r, "\textup{GCK}", rightarrow]                  &\mathcal{AM}(\Omega)
\end{tikzcd}
\end{equation}
\end{thm}
\begin{rem}
The previous theorem holds also when $n$ is even, see \cite[Thm. 4.7]{DDG}.
\end{rem}

\section{The action of Fueter-Sce mapping on Clifford paravector monomials and elementary functions}

The goal of this section is to expand in series some axially elementary  monogenic functions in terms of Clifford-Appell polynomials. In order to achieve this aim we apply the Fueter-Sce map to slice monogenic functions expanded in series. Therefore, it is crucial to figure out how the Fueter-Sce map acts on the paravector monomial $x^k$, with $ k \geq 0$. We will also show that, in the cases of the exponential, trigonometric and hyperbolic functions, the corresponding actions of the Fueter-Sce map and of the generalized CK-extension coincide.

\subsection{Fueter-Sce map and Clifford-Appell polynomials}
In this subsection we will show that the action of the Fueter-Sce map to the to the paravector monomial $x^k$ leads to the Clifford-Appell polynomials, introduced in \cite{CFM, CMF}. First of all we recall the definition of this kind of polynomials. The family of $ \{P_k\}_{k \in \mathbb{N}}$ of Clifford-Appell polynomials is defined as
\begin{equation}
\label{appe}
 P_{k}^{n}(x):= \sum_{s=0}^{k}T_s^k(n)x^{k-s}\overline{x}^{s}, \qquad x \in \mathbb{R}^{n+1},
\end{equation}
where the coefficients $T_s^k(n)$ are given by
\begin{equation}
\label{f0}
T_s^k(n):= \binom{k}{s}\frac{ \left( \frac{n+1}{2}\right)_{k-s} \left( \frac{n-1}{2}\right)_s}{(n)_k}, \qquad n \geq 1 ,
\end{equation}
and $(.)_k$ stands for the Pochammer symbol, i.e. $(a)_s= \frac{\Gamma(a+s)}{\Gamma(a)}$ or $(a)_s= \begin{cases}
a(a+1)(a+2)...(a+s-1) \quad \hbox{for} \quad s>0\\
1  \qquad \qquad \qquad \qquad \qquad \qquad \quad \hbox{for} \quad s=0.
\end{cases}$
It is proved in \cite{CFM} that
\begin{equation}
\label{somma}
\sum_{s=0}^k T_s^k(n)=1.
\end{equation}
The polynomials $P_{k}^{n}(x)$ satisfy the Appell property
\begin{equation}
\label{app}
\frac{1}{2} \overline{\partial}P_{k}^{n}(x_0+ \underline{x})=k P_{k-1}^{n}(x_0+ \underline{x}),
\end{equation}
where $\overline{\partial}=\partial_{x_0}- \sum_{\ell=1}^{n} e_{\ell} \partial_{x_{\ell}}$ is the so-called hypercomplex derivative.
Moreover, the Clifford-Appell polynomials are axially monogenic, see \cite{CFM}. 
Then, by the following fact 
$$ P_k^n(x_0+\underline{0})=x_0^k \sum_{s=0}^k T_k^s(n)=x_0^k,$$
it is clear that
\begin{equation}
\label{GCK2}
P_k^n(x)=GCK[x_0^k].
\end{equation}
Combining this result with Theorem \ref{FG} we obtain the following result.
\begin{thm}
\label{lapla1}
Let $n\geq 1$ be a fixed odd number and $k\geq 0$. Then, for any $x=x_0+ \underline{x}\in \mathbb{R}^{n+1}$ it holds that
\begin{equation}
\label{fund}
\Delta^{\frac{n-1}{2}}_{\mathbb{R}^{n+1}}(x^{n+k-1})= \gamma_n \frac{(n+k-1)!}{k!} P_k^n(x),
\end{equation}
where $\gamma_n:=\frac{(-1)^{\frac{n-1}{2}} 2^{n-1}}{(n-1)!} \left[\Gamma \left(\frac{n+1}{2}\right)\right]^2$.
\end{thm}
\begin{proof}
By the identity $ \partial_{x_0}^{n-1}[x_0^{n-1+k}]= \frac{(n+k-1)!}{k!} x_{0}^{k}$ and the fact that $P_k^n(x)$ is the generalized CK-extension of the monomial $x_0^k$ we get
\begin{eqnarray*}
P_k^n(x)&=& GCK[x_0^k]\\
&=& \frac{k!}{(n+k-1)!} GCK \circ \partial_{x_0}^{n-1}[x_{0}^{n-1+k}].
\end{eqnarray*}
Finally, using Theorem \ref{FG} we obtain
$$\Delta^{\frac{n-1}{2}}_{\mathbb{R}^{n+1}}(x^{n+k-1})= (-1)^{\frac{n-1}{2}}2^{n-1} \left[\Gamma\left(\frac{n+1}{2}\right)\right]^2 \frac{(n+k-1)!}{(n-1)!k!}P_k^n(x).$$
\end{proof}
\begin{rem}
\label{quatapp}
The above result extends to arbitrary dimension those obtained for the quaternionic setting in \cite{DKS}. Indeed, if we consider $n=3$ in \eqref{fund} we get
\begin{eqnarray*}
\Delta_{\mathbb{R}^4} (x^{k+2})&=& -\frac{2(k+2)!}{k!} P^3_k(x)\\
&=&-2 (k+2)(k+1)P^3_k(x),
\end{eqnarray*}
which is exactly the identity obtained in \cite[Rem. 3.9]{DKS}.
\end{rem}

\begin{cor}
\label{lapal2}
Let $x \in \mathbb{R}^{n+1}$ and $n\geq3$ be a fixed odd number. Then
$$ \Delta^{\frac{n-1}{2}}_{\mathbb{R}^{n+1}}(x^{j})= \begin{cases}
\gamma_n  \frac{j!}{(j-n+1)!}P^n_{j+1-n}(x) & \mbox{if } \quad j> n-1,\\
\\
\displaystyle \gamma_n (n-1)! & \mbox{if } \quad j=n-1,\\
\\
\displaystyle 0 & \mbox{if } \quad j< n-1.
\end{cases}
$$
\end{cor}
\begin{proof}
For the case $j> n-1$, it is enough to set $j:= n+k-1$ in formula \eqref{fund}.
\\For the second case we substitute $j=n-1$ in the first case and we obtain
$$
\Delta^{\frac{n-1}{2}}_{\mathbb{R}^{n+1}} (x^{n-1})=4^{\frac{n-1}{2}}(-1)^\frac{n-1}{2} \left(\frac{n-1}{2}\right) \left(\frac{n-1}{2}\right)! \left(\frac{n-3}{2}\right)!=\gamma_n (n-1)!.
$$
\\ Finally, the case $j <n-1$ is trivial because the number of derivatives to perform is more than the degree of the monomial.
\end{proof}

\subsection{Elementary functions}
In Clifford analysis, the building blocks of monogenic functions are the so called Fueter polynomials $\mathcal{P}_{\underline{k}}(x)$ defined by
$$  \mathcal{P}_{\underline{k}}(x)= \frac{1}{k!} \sum_{\sigma \in S_n} z_{j \sigma(1)}...z_{j\sigma(k)},$$
where $S_n$ is the set of permutations of $ \{1,...,n\}$ and $ z_j:= x_{j}-x_{0}e_j$, with $j=1,...,n,$ are the Fueter variables. 
Fueter polynomials play the same role in Clifford analysis as the monomials $ x_0^{\alpha_0}x_{1}^{\alpha_1}x_{2}^{\alpha_2}...x_{n}^{\alpha_n}$ in
analysis of several real variables.
Now, we state the counterpart of Taylor series in a neighbourhood of the origin for monogenic functions, see \cite{BDS}. 
\begin{thm}
\label{monock}
	Let $f$ be a monogenic function in a neighbourhood of $0 \in \mathbb{R}^{m+1}$. Then, the function $f$ admits the series expansion
	$$ f(x)= \sum_{k=0}^\infty \sum_{| \underline{k}|=k} \mathcal{P}_{\underline{k}}(x) a_{\underline{k}}, \qquad \underline{k}=(0,k_1,...,k_n),$$
	where $ \mathcal{P}_{\underline{k}}(x)=CK[x_1^{k_1}...x_{n}^{k_n}]$ and $a_{\underline{k}} \in \mathbb{R}_n$.
\end{thm}

We observe that there is a simpler way to describe monogenic functions defined on axially symmetric domains. This description is achieved by axially monogenic functions, see \cite{Somm1}.

By restricting our analysis to axially monogenic functions, we can substitute the more complicated approach of Fueter polynomials, by the simpler and more accessible approach of Clifford Appell polynomials. The latter are the monogenic counterparts of the powers $x_0^k$ in one-variable real analysis, see \eqref{GCK2}. When restricting our analysis from monogenic to axial functions, the previous Taylor expansion in Theorem \ref{monock} reduces to the following simpler form.
\begin{thm}
\label{newtaylor}
	Let $f$ be an axially monogenic function in a neighbourhood of the origin. Then we can write the function $f$ as
	$$ f(x)= \sum_{k=0}^\infty \frac{f^{(k)}(0)}{k!} P_{k}^{n}(x), \qquad x=x_0+ \underline{x} \in \mathbb{R}^{n+1},$$
	where $P_{k}^n(x)=GCK[x_0^k]$.
\end{thm}
\begin{proof}
	By Theorem \ref{GCK} and the fact that the function $f$ is axially monogenic we have that
	$$ f(x_{0}+ \underline{x})=GCK[f(x_{0})].$$
	By expanding the function $f(x_0)$ in Taylor series and by \eqref{GCK2} we get the statement.  
\end{proof}

In this subsection our aim is to study some axially monogenic elementary functions.
In particular, we shall construct monogenic versions of the exponential,  trigonometric and hyperbolic functions.     
In each of these cases, we are free to use both the Fueter-Sce map and the generalized CK-extension, because in this case the two operators are identical (up to a constant). Indeed, by Theorem \ref{FG} the Fueter-Sce theorem and the generalized CK-extension coincide if
\begin{equation}
\label{eqgck}
GCK[f^{(n-1)}(x_0)]= \frac{1}{\gamma_n}GCK[f(x_0)] .
\end{equation}
Now, since the generalized CK-extension is an isomorphism we get the following condition
$$f^{(n-1)}(x_0)=\frac{1}{\gamma_n}f(x_0).$$ 
This equality is fulfilled by our set of elementary functions i.e., the exponential, the trigonometric and the hyperbolic functions. Below, we list the previous functions in the slice monogenic setting
\begin{defn}
\label{slicemono}
Let $x \in \mathbb{R}^{n+1}$. We define the slice monogenic elementary functions as
$$ e^x:=\sum_{k=0}^{\infty}\frac{x^k}{k!},$$
$$ \sinh(x):=\sum_{k=0}^{\infty}\frac{x^{2k+1}}{(2k+1)!}, $$
$$\cosh(x):=\sum_{k=0}^{\infty}\frac{x^{2k}}{(2k)!}, $$
$$ \sin(x):=\sum_{k=0}^{\infty}\frac{(-1)^k}{(2k+1)!}x^{2k+1},$$
$$ \cos(x):=\sum_{k=0}^{\infty}\frac{(-1)^k}{(2k)!}x^{2k}.$$
\end{defn}
The Fueter-Sce map and the generalized CK-extension coincide, up to a constant, when they are applied to the previous functions. Let $x \in \mathbb{R}^{n+1}$ then we have
\begin{equation}\label{R1}
	\Delta^{\frac{n-1}{2}}_{\mathbb{R}^{n+1}}(e^x)= \gamma_n GCK[e^{x_0}],\end{equation}
$$ \Delta^{\frac{n-1}{2}}_{\mathbb{R}^{n+1}}(\sinh(x))= \gamma_n GCK[\sinh(x_0)],$$
\begin{equation}\label{R3}
	\Delta^{\frac{n-1}{2}}_{\mathbb{R}^{n+1}}(\cosh(x))= \gamma_n GCK[\cosh(x_0)],
\end{equation}
$$ \Delta^{\frac{n-1}{2}}_{\mathbb{R}^{n+1}}(\cos(x))= \gamma_n (-1)^{\frac{n-1}{2}} GCK[\cos(x_0)],$$
$$ \Delta^{\frac{n-1}{2}}_{\mathbb{R}^{n+1}}(\sin(x))= \gamma_n (-1)^{\frac{n-1}{2}} GCK[\sin(x_0)],$$

In literature some elementary monogenic functions in the quaternionic setting are  studied in \cite[Sub. 11.2.3]{G}. The methodology was to extend a complex valued functions to $ \mathbb{R}^{4}$ by means of the slice operator, (see \eqref{slice1}) followed by Laplace operator in four real variables. However, this approach did not use the expansion in series for the slice monogenic functions and hence they did not get a full description of the elementary functions in terms of spherical monogenic polynomials.

Now, we define the elementary monogenic functions by means of the Fueter-Sce map
\begin{defn} Let $n$ be an odd number and $x \in \mathbb{R}^{n+1}$, then we define the monogenic elementary functions as
$$ EXP(x):= \Delta_{\mathbb{R}^{n+1}}^{\frac{n-1}{2}}(e^x),$$
$$ SINH(x):= \Delta_{\mathbb{R}^{n+1}}^{\frac{n-1}{2}}(\sinh(x)),$$
$$ COSH(x):=\Delta_{\mathbb{R}^{n+1}}^{\frac{n-1}{2}}(\cosh(x)),$$
$$ SIN(x):= \Delta_{\mathbb{R}^{n+1}}^{\frac{n-1}{2}}(\sin(x)),$$
$$ COS(x):=\Delta_{\mathbb{R}^{n+1}}^{\frac{n-1}{2}}(\cos(x)).$$
\end{defn}

Now, we write the elementary monogenic functions in terms of Clifford Appell polynomials. The following result is a particular case of Theorem \ref{newtaylor}.
\begin{thm}
Let $n$ be an odd number. Then, for every $x\in\mathbb{R}^{n+1}$, we have
$$ \Delta^{\frac{n-1}{2}}_{\mathbb{R}^{n+1}}(e^x)=\gamma_n\sum_{k=0}^{\infty}\frac{P_k^n(x)}{k!},
$$
$$ \Delta^{\frac{n-1}{2}}_{\mathbb{R}^{n+1}}(\sinh(x))=\gamma_n\sum_{k=1}^{\infty}\frac{P^{n}_{2k+1}(x),}{(2k+1)!},$$
$$ \Delta^{\frac{n-1}{2}}_{\mathbb{R}^{n+1}}(\cosh(x))=\gamma_n \sum_{k=0}^{\infty}\frac{P^{n}_{2k}(x)}{(2k)!},$$
$$
\Delta^{\frac{n-1}{2}}_{\mathbb{R}^{n+1}}(\cos(x))=\gamma_n\sum_{k=0}^{\infty}\frac{(-1)^k}{(2k)!}P^{n}_{2k}(x),
$$
$$\Delta^{\frac{n-1}{2}}_{\mathbb{R}^{n+1}}(\sin(x))=\gamma_n\sum_{k=0}^{\infty}\frac{(-1)^{k}}{(2k+1)!}P^n_{2k+1}(x).
$$
\end{thm}
\begin{proof}
We apply the Fueter-Sce map on the slice monogenic exponential $e^x$. For $n$ an odd fixed number and by Theorem \ref{lapla1} we have
$$\Delta^{\frac{n-1}{2}}_{\mathbb{R}^{n+1}}(e^x)=\sum_{k=n-1}^{\infty}\frac{\Delta^{\frac{n-1}{2}}_{\mathbb{R}^{n+1}}(x^k)}{k!}=\gamma_n\sum_{p=0}^{\infty}\frac{P^n_p(x)}{p!}.$$
Now, we apply the Fueter-Sce map $\Delta^{\frac{n-1}{2}}_{\mathbb{R}^{n+1}}$ to the slice monogenic hyperbolic sinus
$$ \Delta^{\frac{n-1}{2}}_{\mathbb{R}^{n+1}}(\sinh(x))=\sum_{k=0}^{\infty}\frac{\Delta^{\frac{n-1}{2}}_{\mathbb{R}^{n+1}}(x^{2k+1})}{(2k+1)!}=\sum_{k=\lceil \frac{n-2}{2} \rceil}^{\infty}\frac{\Delta^{\frac{n-1}{2}}_{\mathbb{R}^{n+1}}(x^{2k+1})}{(2k+1)!},$$
where $ \lceil. \rceil$ is the ceiling function.
\\By rearranging the index of the sum as $2k+1=n+ \ell-1$, since $n$ is odd we get that $ \ell$ is also odd, and this implies that
$$ \Delta^{\frac{n-1}{2}}_{\mathbb{R}^{n+1}}(\sinh(x))=\sum_{\ell=1, \, \ell \, \,  odd}^{\infty}\frac{\Delta^{\frac{n-1}{2}}_{\mathbb{R}^{n+1}}(x^{n+ \ell-1})}{(n+ \ell-1)!}.$$
Finally, by formula \eqref{fund} we have
$$\Delta^{\frac{n-1}{2}}_{\mathbb{R}^{n+1}}(\sinh(x))=\gamma_n \sum_{k=0}^\infty  \frac{ P^n_{2k+1}(x)}{(2k +1) !}.$$
The expansions of the other monogenic functions in power series follow from similar arguments.
\end{proof}


\begin{prop}
\label{deri}
Let $n $ be an odd fixed number. Then for every $x \in \mathbb{R}^{n+1}$ we have
\begin{itemize}
\item $\overline{\partial} EXP(x)=2 EXP(x),$
\item $\overline{\partial} SINH(x)=2COSH(x),$
\item $\overline{\partial} COSH(x)=2SINH(x),$
\item $\overline{\partial} SIN(x)=2COS(x),$
\item $\overline{\partial} COS(x)=-2SIN(x).$
\end{itemize}
\end{prop}
\begin{proof}
For a generic monogenic function $f$ it is known that $ \overline{\partial}f=2 \partial_{x_0}f$. If we set $f(x):= EXP(x)$, by Remark \ref{R1} we get
\begin{eqnarray*}
\overline{\partial}EXP(x)&=&2 \gamma_n \partial_{x_0} GCK[e^{x_0}]\\
&=&2  \gamma_n GCK[e^{x_0}]\\
&=& 2 EXP(x).
\end{eqnarray*}
The other points follows the same arguments.
\end{proof}

\begin{rem}
It is possible to show the previous result by means of the Appell property of polynomials $P^n_k(x)$ (see \eqref{app}). We show, for instance, the first property
$$\overline{\partial} EXP(x) = 2 \gamma_n\sum_{k=1}^{\infty}\frac{kP^n_{k-1}(x)}{k!}\\
= 2 EXP(x).
$$
\end{rem}

Moreover, we can make the following estimates for the previous  monogenic extensions.
\begin{prop}
Let $n $ be an odd fixed number. Then for every $x \in \mathbb{R}^{n+1}$ we have
\begin{itemize}
\item $|EXP(x)|\leq  |\gamma_n| e^{|x|}, $
\item$|SINH(x)| \leq |\gamma_n|\sinh(|x|),$
\item$|COSH(x)| \leq  |\gamma_n| \cosh(|x|),$
\item$|SIN(x)| \leq |\gamma_n|\sinh(|x|),$
\item $ |COS(x)|\leq |\gamma_n|\cosh(|x|).$
\end{itemize}
\end{prop}
\begin{proof}
We start by proving the first estimate. Formula \eqref{somma} implies that
\begin{equation}
\label{esti}
|P^n_k(x)| \leq  \sum_{s=0}^k T_s^k(n) |x|^{k-s}|x|^s=|x|^k.
\end{equation}
Then we obtain
$$|EXP(x)| \leq |\gamma_n| \sum_{k=0}^{\infty} \frac{|P_{k}^n(x)|}{k!} \leq |\gamma_n| \sum_{k=0}^{\infty} \frac{|x|^k}{k!} =|\gamma_n| e^{|x|}. $$
The other estimates follow by similar computations.
\end{proof}
In this table we summarize all the properties that we have showed for the elementary monogenic functions.
\begin{center}
\begin{tabular}{|c|c|c|c|}
\hline
\rule[-2mm]{-5mm}{-6cm}
& Power series & Hypercomplex derivatives & Estimates \\
\hline
EXP(x)&  $\displaystyle\gamma_n \sum_{k=0}^{\infty}\frac{P_k^n(x)}{k!}$& $2EXP(x)$&$|EXP(x)|\leq |\gamma_n|  e^{|x|}$\\
\hline
SINH(x)&$ \displaystyle\gamma_n\sum_{k=0}^{\infty}\frac{P^{n}_{2k+1}(x)}{(2k+1)!}$&$2COSH(x)$&$|SINH(x)|\leq|\gamma_n|  \sinh(|x|)$\\
\hline
COSH(x)&$\displaystyle\gamma_n\sum_{k=0}^{\infty}\frac{P^{n}_{2k}(x)}{(2k)!}$&$2SINH(x)$& $|COSH(x)| \leq |\gamma_n|  \cosh(|x|)$\\
\hline
SIN(x)& $\displaystyle\gamma_n\sum_{k=0}^{\infty}\frac{(-1)^k}{(2k+1)!}P^{n}_{2k+1}(x)$&$2COS(x)$&$|SIN(x)|\leq |\gamma_n|\sinh(|x|)$\\
\hline
COS(x)&$\displaystyle\gamma_n\sum_{k=0}^{\infty}\frac{(-1)^k}{(2k)!}P^{n}_{2k}(x)$&$-2SIN(x)$&$|COS(x)|\leq |\gamma_n| \cosh(|x|)$\\
\hline
\end{tabular}
\end{center}

Now, we prove formulas that put in relation the monogenic exponential and the hyperbolic functions.  
\begin{prop}
Let $n$ be an odd number. For $x \in \mathbb{R}^{n+1}$ we have
$$ COSH(x)= \frac{EXP(x)+EXP(-x)}{2},$$
$$ SINH(x)= \frac{EXP(x)-EXP(-x)}{2}.$$
\end{prop}
\begin{proof}
It is possible to show this result with the series definitions, but we will use the generalized CK-extension because it is easier. By formula \eqref{R1} and \eqref{R3} we get
\begin{eqnarray*}
COSH(x)&=& \gamma_{n} GCK[\cosh(x_0)]\\
&=& \frac{\gamma_n}{2} \left[GCK[e^{x_0}]+GCK[e^{-x_0}]\right]\\
&=& \frac{EXP(x)+EXP(-x)}{2}.
\end{eqnarray*}
The same arguments follow for the hyperbolic sinus.
\end{proof}
Now, we turn our attention to the following question.
\newline
\newline
\textbf{Problem}
Is it possible to extend the identities which, hold in complex analysis for the trigonometric and hyperbolic functions to Clifford analysis?
\newline
\newline
The key point of the problem is to understand how to multiply monogenic functions.
In this case the pointwise product and the CK-product (denoted by $\odot_{CK}$) do not work. The reason is due to the fact that the pointwise product does not preserve the monogenicity. On the other hand, for a pair of monogenic functions $A(x_0, \underline{x})$ and $B(x_0, \underline{x})$  the CK-product is defined as
$$ A(x_0, \underline{x}) \odot_{CK} B(x_0, \underline{x})=CK[A(0, \underline{x}) \cdot B(0, \underline{x})].$$
As proved in \cite[Prop. 3.7]{ADS2022} (one can easily extend in $ \mathbb{R}^{n+1}$) the CK-product between Clifford-Appell polynomials is given by 
\begin{equation}
\label{Ck}
(P_k^n \odot_{CK} P_s^n)(x)= \frac{c_kc_s}{c_{k+s}}P_{k+s}^n.
\end{equation}
The drawback of this formula is the presence of unsuitable constants $c_k$, depending on the dimension, and the degree $k$, which turn out to be unsuitable for certain computations. 
\\ Therefore, we define a new kind of product between the axially monogenic functions, which makes use of the generalized CK-extension.
\begin{defn}
\label{np}
Let $A(x_0, \underline{x})$ and $B(x_0, \underline{x})$ be axially monogenic functions then
\begin{equation}
\label{np2}
A(x_0, \underline{x}) \odot_{GCK} B(x_0, \underline{x})=GCK[A(x_0,0) \cdot B(x_0,0)].
\end{equation}
\end{defn}
This definition allows the introduction of the following $GCK$ multiplicative inverse.
\begin{defn}
\label{inve}
Let $A(x_0, \underline{x})$ be an axially monogenic function then
$$ [A(x_0, \underline{x})]^{-\odot}=GCK\left[ \frac{1}{A(x_0,0)}\right],$$
whenever $A(x_0,0) \neq 0$.
\end{defn}
In these cases, we restrict the functions to a submanifold of dimension one. We will se that this gives rise to natural properties of the product between Clifford-Appell polynomials, see Lemma \ref{prod} and Proposition \ref{GCK3}.
\\ This new product is fundamental to prove the following relations of the Clifford-Appell trigonometric and hyperbolic functions.
\begin{lem}
\label{idf}
Let $n$ be a fixed odd number. Let $x=x_0+ \underline{x} \in \mathbb{R}^{n+1}$ then
$$ \left(COS(x) \odot_{GCK} COS(x)+SIN(x) \odot_{GCK} SIN(x) \right)=\gamma_n^2,$$
and
$$ \left(COSH(x) \odot_{GCK} COSH(x)-SINH(x) \odot_{GCK} SINH(x)\right)=\gamma_n^2,$$
where $\gamma_n:=\frac{(-1)^{\frac{n-1}{2}}2^{n-1}}{(n-1)!}\left[\Gamma\left(\frac{n+1}{2}\right)\right]^2$.
\end{lem}
\begin{proof}
By Definition \ref{np} we have
\begin{eqnarray}
\nonumber
COS(x) \odot_{GCK} COS(x) &=& \gamma_n^2 \left(GCK[\cos(x_0)]\odot_{GCK}GCK[\cos(x_0)]\right) \\
\label{cos1}
&=& \gamma_n^2GCK[\cos^2(x_0)].
\end{eqnarray}
Similarly
\begin{eqnarray}
\nonumber
SIN(x) \odot_{GCK} SIN(x) &=&\gamma_n^2\left(GCK[\sin(x_0)]\odot_{GCK}GCK[\sin(x_0)]\right) \\
\label{cos2}
&=& \gamma_n^2 GCK[\sin^2(x_0)].
\end{eqnarray}
By making the sum of \eqref{cos1} and \eqref{cos2} we get
\begin{eqnarray*}
COS(x) \odot_{GCK} COS(x)+SIN(x) \odot_{GCK} SIN(x)&=& \gamma_n^2GCK[\cos^2(x_0)]+ \gamma_n^2GCK[\sin^2(x_0)]\\
&=& \gamma_n^2GCK[\cos^2(x_0)+\sin^2(x_0)]\\
&=& \gamma_n^2GCK[1]=\gamma_n^2.
\end{eqnarray*}
\end{proof}

\begin{rem}
By similar arguments it is possible to prove other classical trigonometric identities for the Clifford-Appell monogenic sine and cosine. As well as for the hyperbolic sine and cosine.
\end{rem}

\section{The Kernel and Range of Fueter-Sce mapping}
In classical complex analysis a family of reproducing kernel Hilbert spaces including: Hardy, Fock, Hardy-Sobolev and Dirichlet spaces are studied in \cite{ACS2019}. In this section, we shall extend these spaces to the Clifford setting with the help of the Fueter-Sce map. Moreover, we shall study the kernel and the range of this map when acting on such spaces. Before we introduce our extensions, let us recall that in the complex setting, the Hardy, Fock and the other spaces fit into the following general parallelism.

Given a non-decreasing sequence of real numbers $c=\{c_n\}_{n\in \mathbb{N}}$ with $c_0=1$ we consider the following space of functions
$$ \mathcal{H}(\Omega_c)= \left \{f(z)=\sum_{n=0}^\infty z^n f_n, \quad \{f_n\}_{n \geq 0}\subset \mathbb{C} \quad : \sum_{n=0}^\infty  c_n |f_n|^2 < \infty \right\}, \qquad \Omega_c:=\left\lbrace{z\in\mathbb{C}; \quad \sum_{n=0}^\infty \frac{|z|^{2n}}{c_n}<\infty} \right\rbrace.$$
Each element of the space $\mathcal{H}(\Omega_c)$ absolutely converges on $ \Omega_c$. Indeed by the Cauchy-Schwarz inequality we have
$$  \left(\sum_{n=0}^\infty |z|^n |f_n| \right)^2 = \left(\sum_{n=0}^\infty \frac{|z|^n}{\sqrt{c_n}} \sqrt{c_n} |f_n| \right)^2 \leq \left( \sum_{n=0}^\infty \frac{|z|^{2n}}{c_n}\right) \left( \sum_{n=0}^\infty c_n |f_n|^2\right)< \infty.$$
The space $ \mathcal{H}(\Omega_c)$ can be turned into a Hilbert space with inner product
$$ \langle f, g \rangle_{\mathcal{H}(\Omega_c)}= \sum_{n=0}^\infty c_m f_n \overline{g}_n,$$
where $f(z)= \sum_{n=0}^\infty z^n f_n$, $g(z)= \sum_{n=0}^\infty z^n g_n \in \mathcal{H}(\Omega_c)$.
Moreover, the space $\mathcal{H}(\Omega_c)$ is a reproducing kernel Hilbert space, with reproducing kernel
$$
K_\mathbf{c}(z,w)=\displaystyle \sum_{n=0}^{\infty}\frac{z^n\overline{w}^n}{c_n}, \quad z,w\in \Omega_c.
$$
We refer to Theorem 2.1, Definition 2.2 and Remark 2.3 of \cite{ACS2019} for more details.
\subsection{Range of the Fueter-Sce map}
In this section, we are going to study the action of the Fueter-Sce map on sub-modules of slice monogenic functions of the form
\begin{equation}
\label{slice}
\mathcal{HS}(\Omega_c):=\left\lbrace{\sum_{k=0}^{\infty}x^k\beta_k;\quad \beta_k\in \mathbb{R}_n,\quad \sum_{k=0}^{\infty}c_k|\beta_k|^2<\infty}\right\rbrace, \quad \Omega_c=\left\lbrace{x\in\mathbb{R}^{n+1}; \quad \sum_{k=0}^\infty \frac{|x|^{2k}}{c_k}<+\infty} \right\rbrace,
\end{equation}
where  $c=\{c_k\}_{k \in \mathbb{N}}$ is a non decreasing sequences of real numbers with $c_0=1$. Using the Cauchy-Schwartz inequality, we easily see that functions in $\mathcal{HS}(\Omega_c)$ converge absolutely on $ \Omega_c$.
\\ As we will show, the action of $ \Delta^{\frac{n-1}{2}}_{\mathbb{R}^{n+1}}$ on $ \mathcal{HS}(\Omega_c)$ yields a submodule of axial monogenic functions of the form
\begin{equation}
\label{monos}
\mathcal{HM}(\Omega_b):=\left\lbrace{\sum_{k=0}^{\infty}P_k^n(x)\alpha_k: \quad \alpha_k\in \mathbb{R}_n,\quad \sum_{k=0}^{\infty}b_k|\alpha_k|^2<\infty}\right\rbrace,
\end{equation}
where $ \{b_k\}_{k \in \mathbb{N}_0} \subset \mathbb{R}$ is a suitable non decreasing sequence with $b_0=1$. As for $ \mathcal{HS}(\Omega_c)$, it can be shown that functions in $\mathcal{HM}(\Omega_b)$ converge absolutely on $\Omega_b$.

\begin{defn}
Let us consider $f= \sum_{k=0}^\infty P_{k}^n \alpha_k$ and $g= \sum_{k=0}^\infty P_{k}^n \gamma_k$ in $\mathcal{HM}(\Omega_b)$, then an inner product of $\mathcal{HM}(\Omega_b)$ is defined by
$$ \langle f,g \rangle_{\mathcal{HM}(\Omega_b)}=\sum_{k=0}^\infty b_k\overline{\alpha_k} \gamma_k.$$
Similarly, let $f= \sum_{k=0}^\infty x^k  \beta_k$ and $g= \sum_{k=0}^\infty x^k \delta_k$ in $\mathcal{HS}(\Omega_c)$, then an inner product on $\mathcal{HS}(\Omega_c)$ is defined by
$$ \langle f,g \rangle_{\mathcal{HS}(\Omega_c)}=\sum_{k=0}^\infty c_k\overline{\beta_k} \delta_k.$$
In particular the associate norms are given by 

$$ ||f||_{\mathcal{HS}(\Omega_c)}=\sqrt{\hbox{Re}\left(\langle f,f \rangle_{\mathcal{HS}(\Omega_c)}\right)}, \qquad ||f||_{\mathcal{HM}(\Omega_b)}=\sqrt{\hbox{Re}\left(\langle f,f \rangle_{\mathcal{HM}(\Omega_b)} \right)},$$

\end{defn}
where $\hbox{Re}\left(\langle .\rangle \right)$ denotes the scalar part of the inner product.

\medskip

It turns out  that all the evaluation mapping on $ \mathcal{HM}(\Omega_b)$ are continuous. This implies the existence of a reproducing kernel for the space $ \mathcal{HM}(\Omega_b)$.
\begin{thm}
\label{Chap}
The evaluation map on $ \mathcal{HM}(\Omega_b)$ is bounded. In particular for every function $f \in \mathcal{HM}(\Omega_b)$ we have the following estimate
\begin{equation}
	\displaystyle |f(x)|\leq \left(2^n\sum_{k=0}^\infty \frac{|x|^{2k}}{b_k}\right)^{1/2}||f||_{\mathcal{HM}(\Omega_b)} ; \quad \forall x\in \Omega_b.
\end{equation}
This means that the evaluation map is continuous on $\mathcal{HM}(\Omega_b)$, which is sufficient to endow $\mathcal{HM}(\Omega_b)$ with a reproducing kernel Hilbert space structure. In particular, the reproducing kernel of $ \mathcal{HM}(\Omega_b)$ is given by 
\begin{equation}
\label{gen}
K_{b}(x,y)=\sum_{k=0}^{\infty}\frac{P_k^n(x) \overline{P_k^n(y)}}{b_k}, \quad x,y\in\Omega_b.
\end{equation}
\end{thm}
\begin{proof}
Let $f(x)=\sum_{k=0}^\infty P_k^n(x) \alpha_k$. We first recall that $|ab| \leq 2^{\frac{n}{2}} |a| |b|$ for every $a$, $b \in \mathbb{R}_n$. Then, by the Cauchy-Schwarz inequality and the fact that $|P_k^n(x)|\leq |x|^k$ we get
\begin{eqnarray*}
	|f(x)| & \leq & 2^{\frac{n}{2}}\sum_{k=0}^{\infty}|P_k^n(x)||\alpha_k|=2^{\frac{n}{2}}\sum_{k=0}^{\infty}\frac{|P_k^n(x)|}{\sqrt{b_k}}|\alpha_k| \sqrt{b_k}\\
	& \leq & \left(2^n \sum_{k=0}^\infty \frac{|P_k^n(x)|^2}{b_k}\right)^{\frac{1}{2}} \left( \sum_{k=0}^\infty b_k |\alpha_k|^2\right)^{\frac{1}{2}}\\
	&\leq& \left(2^n\sum_{k=0}^\infty \frac{|x|^{2k}}{b_k}\right)^{1/2}||f||_{\mathcal{HM}(\Omega_b)}.
\end{eqnarray*}
Since $f \in \mathcal{HM}(\Omega_b)$, it is clear that the quantity $\sum_{k=0}^\infty \frac{|x|^{2k}}{b_k}$ is bounded for all $x \in \Omega_b$. Thus the evaluation map is bounded and therefore $ \mathcal{HM}(\Omega_b)$ is a reproducing kernel Hilbert space.
\\ Let us now show that the series in \eqref{gen} is indeed the reproducing kernel for $ \mathcal{HM}(\Omega_b)$. First we show that it is convergent.
Using the Cauchy-Schwarz inequality and the fact that $|P_k^n(x)|\leq |x|^k$ for $x$, $y \in \Omega_b$ we obtain for $x$, $y \in \Omega_b$ that
\begin{eqnarray*}
| K_b(x,y)| & \leq &  2^n \sum_{k=0}^\infty  \frac{|P_k^n(x)||P_k^n(y)|}{b_k}\\
& \leq & 2^n \sum_{k=0}^\infty  \frac{|x|^k |y|^k}{b_k}\\
& \leq & 2^n \left( \sum_{k=0}^{\infty} \frac{|x|^{2k}}{b_k}\right)\left( \sum_{k=0}^{\infty} \frac{|y|^{2k}}{b_k}\right) < \infty.
\end{eqnarray*}

To prove the reproducing kernel property, let us fix $y \in \Omega_b$, so we have
	$$ K_y(x)= \sum_{k=0}^\infty P_k^n(x) \alpha_k(y), \qquad \alpha_k(y):= \frac{\overline{P_k^n(y)}}{b_k}.$$
	We observe that the function $K_y(x)$ belongs to $\mathcal{HM}(\Omega_b)$. Indeed, if $y \in \Omega_b$, by inequality \eqref{esti} we get
	$$ \sum_{k=0}^\infty b_k |\alpha_k(y)|^2= \sum_{k=0}^\infty \frac{|P_k^n(y)|^2}{b_k} \leq \sum_{k=0}^\infty \frac{|y|^{2k}}{b_k}< \infty.$$
Let us consider $f(x)= \sum_{k=0}^\infty P_k^n(x) \beta_k \in \mathcal{HM}(\Omega_b)$, with $ \{\beta_k\}_{k \in \mathbb{N}_0} \subset \mathbb{R}_n$, then we have
	$$ \langle K_y, f \rangle_{\mathcal{HM}(\Omega_b)}= \sum_{k=0}^\infty b_k \overline{\alpha_k(y)} \beta_k=\sum_{k=0}^\infty P_k^n(y) \beta_k=f(y), \qquad \forall y \in \Omega_b.$$
\end{proof}
\begin{thm}
\label{main}
Let $n$ be a fixed odd number and $c=\{c_k\}_{k \in \mathbb{N}}$ be a non decreasing sequence with $c_0=1$. Then
$$ \Delta^{\frac{n-1}{2}}_{\mathbb{R}^{n+1}}\left( \mathcal{HS}(\Omega_c)\right) \subset \mathcal{HM}(\Omega_b),$$
where the sequence $b=\{b_k\}_{k \in \mathbb{N}_0}$ is given by
\begin{equation}
\label{sequence}
b_k:= \frac{c_{k+n-1}(k!)^2}{[(n+k-1)!]^2}, \quad k \in \mathbb{N}_0.
\end{equation}
Moreover, for any $f\in\mathcal{HS}(\Omega_c)$  we have
$$ \|\Delta_{\mathbb{R}^{n+1}}^{\frac{n-1}{2}}f\|_{\mathcal{HM}(\Omega_b)}^2=\gamma^2_n\left(||f||^2_{\mathcal{HS}(\Omega_c)}-\sum_{k=0}^{n-2}c_k \frac{|f^{(k)}(0)|^2}{[k!]^2}\right),
$$
where $\gamma_n:= \frac{(-1)^{\frac{n-1}{2}} 2^{n-1}}{(n-1)!} \left[\Gamma \left( \frac{n+1}{2}\right)\right]^2.$
\end{thm}
\begin{proof}
Let $g\in\Delta_{\mathbb{R}^{n+1}}^{\frac{n-1}{2}}\left(\mathcal{HS}(\Omega_c)\right)$, thus there exists $f\in \mathcal{HS}(\Omega_c)$ such that $g= \Delta_{\mathbb{R}^{n+1}}^{\frac{n-1}{2}} f$. Writing
$$ f(x)=\sum_{p=0}^\infty x^p\alpha_p, \quad \hbox{where} \ \quad \{\alpha_p\}_{p \in \mathbb{N}_0} \subset \mathbb{R}_n \quad \hbox{and} \quad \sum_{p=0}^{\infty}|\alpha_p|^2c_p<+\infty,$$
we obtain
\begin{align*}
g(x) &= \Delta_{\mathbb{R}^{n+1}}^{\frac{n-1}{2}}f(x)  \\
&=\sum_{s=n-1}^{\infty} \Delta^{\frac{n-1}{2}}_{\mathbb{R}^{n+1}}(x^s)\alpha_s\\
&=\sum_{k=0}^{\infty} \Delta^{\frac{n-1}{2}}_{\mathbb{R}^{n+1}}(x^{k+n-1})\alpha_{k+n-1}\\
&= \gamma_n \sum_{k=0}^{\infty} \frac{(n+k-1)!}{k!}P^n_k(x)\alpha_{k+n-1}\\
&:=\sum_{k=0}^{\infty} P^n_k(x)\beta_k;
\end{align*}
where the coefficients $\{\beta_k\}_{k \in \mathbb{N}_0}$ are given by
\begin{equation}
\label{ali1}
\beta_k:= \gamma_n \frac{(n+k-1)!}{k!} \alpha_{k+n-1}, \quad k \in \mathbb{N}_0.
\end{equation}
Now, we set \begin{equation} \label{b}
b_k:= \frac{c_{k+n-1}(k!)^2}{[(n+k-1)!]^2};\quad \forall k\geq 0.
\end{equation}

Then, recalling the definition of the sub-module of monogenic functions $\mathcal{HM}(\Omega_b)$ (see \eqref{monos}) we have the following equalities

\begin{align*}
\|\Delta_{\mathbb{R}^{n+1}}^{\frac{n-1}{2}}f\|^2_{\mathcal{HM}(\Omega_b)} &=\sum_{k=0}^{\infty}b_k|\beta_k|^2  \\
&= \gamma_n^2 \sum_{k=0}^{\infty}b_k \frac{[(n+k-1)!]^2}{(k!)^2}|\alpha_{k+n-1}|^2 \\
&= \gamma_n^2\sum_{k=0}^{\infty}c_{k+n-1}|\alpha_{k+n-1}|^2\\
&=\gamma_n^2 \sum_{p=n-1}^{\infty} c_p|\alpha_p|^2 .
\end{align*}
Now, we recall the the coefficients $ \{\alpha_{k}\}_{k \geq 0}$ appearing in the Taylor series of the slice monogenic functions $f$ are given by 
$$\displaystyle  \begin{cases}
\vspace*{-2mm}
\alpha_0:=f(0)\\
\vspace*{-2mm}
.\\
\vspace*{-2mm}
.\\
\vspace*{-2mm}
.\\
\alpha_{n-2}:=\frac{f^{(n-2)}(0)}{(n-2)!}.
\end{cases}
$$
Finally, we conclude that
$$\|\Delta_{\mathbb{R}^{n+1}}^{\frac{n-1}{2}} f\|^2_{\mathcal{HM}(\Omega_b)}=\gamma_n^2 \left(||f||^2_{\mathcal{HS}(\Omega_c)}-\sum_{k=0}^{n-2}c_k\frac{|f^{(k)}(0)|^2}{[k!]^2}\right).$$
\end{proof}
\begin{prop}
The Fueter-Sce map, which goes from the space $\mathcal{HS}(\Omega_c)$ to $ \mathcal{HM}(\Omega_b)$ is surjective, i.e.,
$$\Delta_{\mathbb{R}^{n+1}}^{\frac{n-1}{2}}(\mathcal{HS}(\Omega_c))=\mathcal{HM}(\Omega_b).$$
\end{prop}
\begin{proof}
We have to show that for any $g \in \mathcal{HM}(\Omega_b)$ there exists a function $f \in \mathcal{HS}(\Omega_c)$ such that
$$ g(x)= \Delta_{\mathbb{R}^{n+1}}^{\frac{n-1}{2}} f(x).$$
By \eqref{monos} we know that 
\begin{equation}
\label{formulan}
g(x)= \sum_{k=0}^\infty P_{k}^n(x) \alpha_k, \qquad \sum_{k=0}^\infty b_k | \alpha_k|^2 < \infty,
\end{equation}
where $ \{\alpha_k\}_{k \geq 0} \subset\mathbb{R}_n$. By Theorem \ref{lapla1} we have
\begin{eqnarray*}
g(x)&=& \sum_{k=0}^\infty P_{k}^n(x) \alpha_k\\
&=& \sum_{k=0}^\infty  \frac{k! \Delta_{\mathbb{R}^{n+1}}^{\frac{n-1}{2}}(x^{n+k-1})}{\gamma_n (n+k-1)!} \alpha_{k}\\
&=& \Delta_{\mathbb{R}^{n+1}}^{\frac{n-1}{2}} \left(\sum_{k=n-1}^{\infty} x^{k} \frac{(k+1-n)!}{\gamma_n k!} \alpha_{k+1-n} \right)\\
&:=& \Delta_{\mathbb{R}^{n+1}}^{\frac{n-1}{2}} f(x),
\end{eqnarray*}
where $ \beta_k:=\frac{(k+1-n)!}{\gamma_n k!} \alpha_{k+1-n}$. In order to prove that $f \in  \mathcal{HS}(\Omega_c)$ we have to show $ \sum_{k=0}^\infty c_k| \beta_k|^2 < \infty$. By Theorem \ref{main} and the condition of convergence in \eqref{formulan} we have
\begin{eqnarray*}
\sum_{k=0}^\infty c_k| \beta_k|^2 &=& \sum_{k=n-1}^\infty c_k \frac{[(k+1-n)!]^2}{\gamma_n^2 [k!]^2} | \alpha_{k+1-n}|^2\\
&=& \sum_{k=0}^\infty c_{k+1-n} \frac{[k!]^2}{[(k+1-n)!]^2}| \alpha_{k}|^2\\
&=& \sum_{k=0}^\infty b_k | \alpha_k|^2 < \infty,
\end{eqnarray*}
which proves the result 
\end{proof}

If we consider the space
$$ \mathcal{HS}_{c}^{0}:= \left\{f \in \mathcal{HS}(\Omega_c), \, f(0)=...= f^{(n-2)}(0)=0\right\},$$
we have the following result.
\begin{cor}
Let $n$ be an odd number. The Fueter-Sce mapping $ \Delta_{\mathbb{R}^{n+1}}^{\frac{n-1}{2}}$ defines an isometry (up to a constant) from $\mathcal{HS}_{c}^{0}$ onto $ \mathcal{HM}$. In particular, 
$$ \| \Delta_{\mathbb{R}^{n+1}}^{\frac{n-1}{2}}(f) \|_{\mathcal{HM}(\Omega_b)}= \gamma_n ||f||_{\mathcal{HS}_c^0(\Omega_c)}.$$
Moreover, the inner product is preserved i.e.,
$$ \langle \Delta_{\mathbb{R}^{n+1}}^{\frac{n-1}{2}} f, \Delta_{\mathbb{R}^{n+1}}^{\frac{n-1}{2}} g \rangle_{\mathcal{HM}(\Omega_b)} = \gamma_n^2 \langle f,g \rangle_{\mathcal{HS}_c^0(\Omega_c)}.$$
\end{cor}

\begin{rem}
The above result constitutes an extension to all odd dimensions $n$ of \cite[Thm. 6.1]{ADS2022} where similar results were proven in the quaternionic case, i.e., for $n=3$. 
\end{rem}
Now, we show some examples of slice monogenic function spaces and their Fueter-Sce mapping ranges, the associated sequences $c$ and $b$ and the Fueter-Sce mapping norm. In the sequel we denote by $ \mathbb{B}$ the unit ball defined in the following way
$$ \mathbb{B}:= \left\{x \in \mathbb{R}^{n+1}\, : \, \sum_{i=0}^n x_{i}^2 <1 \right\}.$$
\begin{ex}[Hardy space]
\label{01}
In this case we consider $c_k=1$. Then we define the slice monogenic Hardy space as
$$\mathsf{HA}(\mathbb{B}):= \left\{\sum_{k=0}^\infty x^k \alpha_k: \quad \alpha_k \in \mathbb{R}_n, \quad \sum_{k=0}^\infty |\alpha_k|^2< \infty \right\}.$$
The action of the Fueter-Sce map on $\mathsf{HA}(\mathbb{B})$ gives the axially monogenic Hardy space. By Theorem \ref{main} we know that this space is given by
\begin{equation}
\label{space1}
\mathcal{HA}_n(\Omega_b):=\left\lbrace{\sum_{k=0}^{\infty}P_k^n(x)\alpha_k, \quad \alpha_k\in \mathbb{R}_n, \quad \sum_{k=0}^{\infty}b_k|\alpha_k|^2<\infty}\right\rbrace,
\end{equation}
where 
$$ b_k= \frac{[k!]^2}{[(k+n-1)!]^2}.$$
By \eqref{monos}, the ratio test and the fact that
$\lim_{k \to \infty} \frac{b_k}{b_{k+1}}=1,$
we get that the domain of convergence $\Omega_b$ coincides with the unit ball $ \mathbb{B} \subset \mathbb{R}^{n+1}$. 
Moreover, by Theorem \ref{main} we have the following equality
$$ \|\Delta_{\mathbb{R}^{n+1}}^{\frac{n-1}{2}} f\|_{\mathcal{HA}_{n}(\mathbb{B})}^2=\gamma_n^2\left(||f||_{\mathsf{HA}(\mathbb{B})}^{2}-\sum_{k=0}^{n-2}\frac{|f^{k}(0)|^2}{[k!]^2}\right).$$
\end{ex}

\begin{ex}[Bergman space]
In this case we consider $c_k=\frac{1}{k+1}$, i.e. the slice monogenic Bergman space is defined as
$$\mathsf{A}(\mathbb{B}):= \left\{\sum_{k=0}^\infty x^k \alpha_k: \quad \alpha_k \in \mathbb{R}_n,  \quad \sum_{k=0}^\infty \frac{|\alpha_k|^2}{k+1}< \infty\right\}.$$
The action of the Fueter-Sce map on $\mathsf{A}(\mathbb{B})$ gives the axially monogenic Bergman space. By Theorem \ref{main}, this space is given by
$$ \mathcal{A}(\Omega_b):=\left\lbrace{\sum_{k=0}^{\infty}P_k^n(x)\alpha_k,\quad \alpha_k\in \mathbb{R}_n,\quad \sum_{k=0}^{\infty}b_k|\alpha_k|^2<\infty}\right\rbrace,$$
where 
$$ b_k=\frac{(k!)^2}{(k+n)!(k+n-1)!}.$$
Since $ \lim_{k \to \infty}\frac{b_k}{b_{k+1}}=1$, by the ratio test and \eqref{monos} we conclude that $\Omega_b= \mathbb{B}$. 
Furthermore, by Theorem \ref{main} we have the following equality
$$ ||\Delta_{\mathbb{R}^{n+1}}^{\frac{n-1}{2}} f||_{\mathcal{A}(\mathbb{B})}^2=\gamma_n^2 \left(||f||_{\mathsf{A}(\mathbb{B})}^{2}-\sum_{k=0}^{n-2}\frac{|f^{k}(0)|^2}{(k+1)!k!}\right).$$
\end{ex}
\begin{ex}[Dirichlet space]
In this case we consider $c_k=k$. Thus, we define the slice monogenic Dirichlet space as
$$\mathsf{D}(\mathbb{B}):= \left\{\sum_{k=0}^\infty x^k \alpha_k: \quad \alpha_k \in \mathbb{R}_n , \quad \sum_{k=1}^\infty k|\alpha_k|^2< \infty \right\}.$$
The action of the Fueter-Sce map on $\mathsf{D}(\mathbb{B})$ gives the axially monogenic Dirichlet space. By Theorem \ref{main} we know that this space is given by
$$ \mathcal{D}(\Omega_b):=\left\lbrace{\sum_{k=0}^{\infty}P_k^n(x)\alpha_k,\quad \alpha_k\in \mathbb{R}_n, \quad \sum_{k=0}^{\infty}b_k|\alpha_k|^2<\infty}\right\rbrace,$$
where 
$$ b_k=\frac{(k!)^2}{(k+n-1)!(k+n-2)!}.$$
We observe that $ \Omega_b=\mathbb{B}$. Indeed. it suffices to use the ratio test and the fact that $ \lim_{k \to \infty} \frac{b_k}{b_{k+1}}=1$. 
Moreover, by Theorem \ref{main}, we have the following equality
$$ ||\Delta_{\mathbb{R}^{n+1}}^{\frac{n-1}{2}} f||_{\mathcal{D}(\mathbb{B})}^2=\gamma_n^2 \left(||f||_{\mathsf{D}(\mathbb{B})}^{2}-\sum_{k=1}^{n-2}\frac{|f^{(k)}(0)|^2}{k!(k-1)!}\right).$$
\end{ex}

\begin{ex}[Fock space]
\label{02} 
In this case we have $c_k=k!$. Therefore we define the slice monogenic Fock space as
$$\mathsf{SB}(\mathbb{R}^{n+1}):= \left\{ \sum_{k=0}^\infty x^k \alpha_k: \quad \alpha_k \in \mathbb{R}_n,  \quad \sum_{k=0}^\infty k!|\alpha_k|^2< \infty\right\}.$$
The action of the Fueter-Sce map on $\mathsf{SB}(\mathbb{B})$ gives the axially monogenic Fock space. By Theorem \ref{main} we know that this space is defined by
\begin{equation}
\label{space2}
\mathcal{SB}(\Omega_b):=\left\lbrace{\sum_{k=0}^{\infty}P_k^n(x)\alpha_k: \quad \alpha_k\in \mathbb{R}_n, \quad \sum_{k=0}^{\infty}b_k|\alpha_k|^2<\infty}\right\rbrace,
\end{equation}
where 

$$ b_k=\frac{(k!)^2}{(k+n-1)!}.$$
In this case, $ \lim_{k \to \infty} \frac{b_{k}}{b_{k+1}}=0$, hence by the ratio test we obtain that $\Omega_b= \mathbb{R}^{n+1}$. 
Moreover, by Theorem \ref{main} we have the following equality
$$ \|\Delta_{\mathbb{R}^{n+1}}^{\frac{n-1}{2}} f\|_{\mathcal{SB}(\mathbb{R}^{n+1})}^2=\gamma_n^2\left(\|f\|_{\mathsf{SB}(\mathbb{R}^{n+1})}^{2}-\sum_{k=0}^{n-2}\frac{|f^{(k)}(0)|^2}{k!}\right).$$
	
\end{ex}

In the table below, we summarize the most relevant information concerning the previous examples.  

\begin{center}
	\begin{tabular}{|c|c|c|c|}
		\hline
		\hline
		\rule[-1mm]{-5mm}{-2cm}
		& $\displaystyle c_k$ & $\displaystyle b_k$ & $\| \displaystyle\Delta_{\mathbb{R}^{n+1}}^{\frac{n-1}{2}} f \|_{\mathcal{HM}_b(\Omega_b)}$\\
		\hline
		Hardy&$\displaystyle 1$& $ \displaystyle \frac{[k!]^2}{[(k+n-1)!]^2}$& $\displaystyle\gamma_n^2\left(||f||_{\mathsf{HA}(\mathbb{B})}^{2}-\sum_{k=0}^{n-2}\frac{|f^{k}(0)|^2}{[k!]^2}\right)$\\
		\hline
		Fock&$ \displaystyle k!$&$ \displaystyle \frac{(k!)^2}{(k+n-1)!}$&$\displaystyle \gamma_n^2\left(\|f\|_{\mathsf{SB}(\mathbb{R}^{n+1})}^{2}-\sum_{k=0}^{n-2}\frac{|f^{k}(0)|^2}{k!}\right)$\\
		\hline
		Bergman&$ \displaystyle \frac{1}{k+1}$& $\displaystyle \frac{(k!)^2}{(k+n)!(k+n-1)!}$&$\displaystyle \gamma_n^2\left(||f||_{\mathsf{A}(\mathbb{B})}^{2}-\sum_{k=0}^{n-2}\frac{|f^{k}(0)|^2}{(k+1)!k!}\right)$\\
		\hline
		Dirichlet& $\displaystyle k$& $\displaystyle \frac{(k!)^2}{(k+n-1)!(k+n-2)!}$ &$\displaystyle \gamma_n^2\left(||f||_{\mathsf{D}(\mathbb{B})}^{2}-\sum_{k=1}^{n-2}\frac{|f^{k}(0)|^2}{k!(k-1)!}\right)$\\
		\hline
		\hline
	\end{tabular}
\end{center}

\subsection{Kernel of the Fueter-Sce map}
In this subsection we turn our attention to the kernel of the Fueter-Sce map,  i.e.,

$$ \ker(\Delta_{\mathbb{R}^{n+1}}^{\frac{n-1}{2}}):=\left\lbrace{f\in \mathcal{SM}(\Omega),\quad \Delta^{\frac{n-1}{2}}_{\mathbb{R}^{n+1}}f(x)=0}\right\rbrace,$$

where the set $ \Omega$ is an axially symmetric slice domain in $ \mathbb{R}^{n+1}$.
\begin{thm}\label{kerth}
Let $n$ be a fixed odd number and let $\Omega$ be an axially symmetric slice domain on $\mathbb{R}^{n+1}$. Then,
a slice monogenic function $f$ belongs to $\ker(\Delta_{\mathbb{R}^{n+1}}^{\frac{n-1}{2}})$ if and only if $f$ is a $ \mathbb{R}_n$-valued polynomial of degree $n-2$ in the variable $x$, i.e.
\begin{equation}
f(x)=\sum_{k=0}^{n-2}x^k\alpha_{k}, \qquad \alpha_k \in \mathbb{R}_n.
\end{equation}
\end{thm}
\begin{proof}
Without lost of generality we assume that the axially symmetric slice domain contains the origin. If this is not the case one can proceed using a translation argument.
First, it is clear that if $f$ is a function of the form $$f(x)=\alpha_0+\alpha_1 x+...+\alpha_{n-2}x^{n-2},$$ then $f$ is slice monogenic and belongs to the kernel of $\Delta_{\mathbb{R}^{n+1}}^{\frac{n-1}{2}}$.
\\Conversely, let $f\in \ker(\Delta_{\mathbb{R}^{n+1}}^{\frac{n-1}{2}})$.
We know by the series expansion theorem for slice monogenic functions that $f$ can be expanded as
\begin{equation}
\label{somma1}
f(x)= \sum_{k=0}^{\infty}x^k \alpha_k, \quad \forall x\in \Omega.
\end{equation}

Now, we apply the Fueter-Sce map $\Delta_{\mathbb{R}^{n+1}}^{\frac{n-1}{2}}$ to \eqref{somma1}. By Corollary \ref{lapal2} and after some manipulations we get

\begin{equation}
\Delta^{\frac{n-1}{2}}_{\mathbb{R}^{n+1}} f(x)= \sum_{\ell=0}^{\infty} P^n_{\ell}(x)\beta_\ell^n;
\end{equation}

where the coefficients of this polynomial are given by (see \eqref{ali1})
\begin{equation}
\beta_\ell^n:= \gamma_n  \frac{(n+ \ell-1)!}{\ell!}\alpha_{\ell+n-1}, \quad \ell \in \mathbb{N}_0.
\end{equation}

Since $\Delta^{\frac{n-1}{2}}_{\mathbb{R}^{n+1}}f=0$, for any $x \in \Omega$ we have
\begin{equation}
\sum_{\ell=0}^{\infty}P^n_{\ell}(x)\beta_\ell^n =0.
\end{equation}
So, by restricting to the real line we get
\begin{equation}
\displaystyle \sum_{\ell=0}^{\infty}x_0^{\ell}\beta_\ell^n =0.
\end{equation}
In particular, this shows for every $\ell=0,1,...$ that

\begin{equation}
\alpha_{\ell+n-1}=0.
\end{equation}
Hence, we have $\alpha_k=0$ for every $k=n-1,n,....$ . So, this implies that
$$f(x)= \sum_{k=0}^{n-2}x^k \alpha_k, \quad \forall x\in \Omega,$$
which finishes the proof.
\end{proof}
\begin{rem}
We used the Clifford-Appell approach for the proof of Theorem \ref{kerth} since it will be of interest in the sequel and it will be also useful to characterize the functions that belong to the kernel of the operator $\partial$, see \cite{ANST}. However, it is possible to show this result using Theorem \ref{FG}. Indeed $ \Delta_{\mathbb{R}^{n+1}}^{\frac{n-1}{2}}f=0$ if and only if
$$ GCK[f^{(n-1)}(x_0)]=0.$$
This holds if and only if
$$ f^{(n-1)}(x_0)=0.$$
Finally, since $f$ si defined in a piecewise connected open set $ \mathbb{R}$, it can be uniquely extended to a holomorphic function to a connected set in $ \mathbb{R}^2$. Thus $f$ is a polynomial of degree at most $n-2$, i.e.
$$ f(x)= \sum_{k=0}^{n-2}x^k \alpha_k, \quad  \{\alpha_k\}_{0 \leq k \leq n-2}  \subset \mathbb{R}_n.$$
\end{rem}

\section{The Clifford-Appell-Fock space in $ \mathbb{R}^{n+1}$}
In this section our aim is to study an extension of the Fock space to the monogenic setting. In the complex setting, an entire function $ f(z)= \sum_{n=0}^\infty z^n a_n$ belongs to the complex Fock space if its Taylor coefficient satisfy $ \sum_{n=0}^\infty n! |a_n|^2< \infty$, see \cite{F}.
A Fock space in the slice monogenic setting was studied in \cite{ACSS}.
In this work, we will further extend the Fock space in the axially monogenic setting. 
To do this we will make use of the generalized CK-extension. The main advantages of this tool, with respect to the Fueter-Sce theorem, is that it is an isomorphism. Therefore it preserves the structure of the complex Fock space, see \eqref{GCK2}.
\\ Let us define the Clifford-Appell-Fock space in $ \mathbb{R}^{n+1}$. This space is defined, for a fixed odd number $n$, as
\begin{equation}
\label{space21}
\mathcal{F}(\mathbb{R}^{n+1}):=\left\lbrace{ \sum_{k=0}^{\infty}P_k^n(x)\alpha_k: \quad\alpha_k\in\mathbb{R}_n,\quad\sum_{k=0}^{\infty}k!|\alpha_k|^2<\infty}\right\rbrace.
\end{equation}
It follows from the classical case that the Clifford-Appell-Fock space is defined in the whole space, because for any $x \in \mathbb{R}^{n+1}$ we know that $| P_k^n(x)| \leq |x|^k$.
\begin{rem}
For any $k \geq 0$ and $n$ fixed odd number we have that $ \frac{k!}{(k+n-1)!}  \leq k!$. From this it follows that he Fock space introduced in \eqref{space21} is contained in the Fock space defined in \eqref{space2} .
\end{rem}
As we already saw, this kind of space can be endowed with a Hilbert module structure as follows.
\begin{defn}
Let $n$ be a fixed odd number. Let us consider $f(x)= \sum_{k=0}^\infty P_k^n(x) \alpha_k$ and $ g(x)=\sum_{k=0}^\infty P_k^n(x) \beta_k$ in $\mathcal{F}(\mathbb{R}^{n+1})$. We define the scalar product as
$$ \langle f,g \rangle_{\mathcal{F}(\mathbb{R}^{n+1})}= \sum_{k=0}^\infty k! \overline{\alpha}_k \beta_k.$$
\end{defn}

\begin{thm}
\label{eval}
Let $n$ be a fixed odd number. For any $f \in  \mathcal{F}(\mathbb{R}^{n+1})$ and $x \in \mathbb{R}^{n+1}$, we have
$$ |f(x)| \leq 2^{\frac{n}{2}} e^{\frac{|x|^2}{2}} \| f \|_{\mathcal{F}(\mathbb{R}^{n+1})}.$$
Moreover, the reproducing kernel of the Clifford-Appell-Fock space $\mathcal{F}(\mathbb{R}^{n+1})$ is given by the convergent series
\begin{equation}
	\label{rkh}
	K_{\mathcal{F}}(x,y)=\sum_{k=0}^{\infty}\frac{P_k^n(x) \overline{P_k^n(y)}}{k!}, \quad x,y\in\mathbb{R}^{n+1}.
\end{equation}
\end{thm}
\begin{proof}
It is enough to put $b_{k}= k!$ and $ \Omega_b= \mathbb{R}^{n+1}$ in Theorem \ref{Chap} . 
\end{proof}

\begin{rem}
Currently we do not have a method for computing the sum in \eqref{rkh}. Nevertheless, if we restrict $x$ and $y$ to the real line, i.e. $x_0$,$y_0 \in \mathbb{R}$ we get
$$ K_{\mathcal{F}}(x_0,y_0)=e^{x_0y_0},$$
which is the reproducing kernel of the classical Fock space.
\end{rem}
Before to extend the creation and annihilation operators in this setting, we prove the following fundamental result.
\begin{lem}
\label{prod}
Let $n$ be a fixed odd number and $ \ell, k \geq 0$. Then for any $x=x_0+ \underline{x} \in \mathbb{R}^{n+1}$ we have
\begin{equation}
\label{dip}
P_{\ell}^n(x) \odot_{GCK} P_k^n(x)=P_{k+ \ell}^n(x).
\end{equation}
\end{lem}
\begin{proof}
It follows by formulas \eqref{GCK2} and \eqref{np2}
\begin{eqnarray*}
P_{\ell}^n(x) \odot_{GCK} P_k^n(x)&=&GCK[x_{0}^\ell x_0^{k}]\\
&=& GCK[x_{0}^{\ell+k}]\\
&=& P^n_{k+ \ell}(x).
\end{eqnarray*}
\end{proof}
Using the definition of GCK-product (see \ref{np2}) we can now extend to the Clifford setting the classical notions of creation and annihilation operators.
\begin{defn}
\label{cre}
The creation operator acting on a function $f \in \mathcal{F}(\mathbb{R}^{n+1}) $ is defined as
\begin{equation}
\mathcal{M}_{P_1^n}(f):= P_1^n \odot_{GCK} f,
\end{equation}
where $P_1^n(x)= \frac{1}{ 2n} \left( (n+1)x+ (n-1) \bar{x}\right)$.
\end{defn}
Let $g(x)= \sum_{k=0}^\infty P_k^n(x) \beta_k$, Then by Lemma \ref{prod} we get
\begin{eqnarray}
\nonumber
\left(\mathcal{M}_{P_1^n} g\right)(x) &=& \sum_{k=0}^\infty \mathcal{M}_{P_1^n}(P_k^n(x)) \beta_k\\
\label{shift}
&=& \sum_{k=0}^\infty P^n_{k+1}(x) \beta_k.
\end{eqnarray}
This means that the operator $\mathcal{M}_{P_1^n}$ can be considered as a shift operator with respect to the Clifford-Appell system $\{P_k^n\}_{k \geq  0}$.
\begin{defn}
We define the annihilation operator as the hypercomplex derivative, i.e.
\begin{equation}
\frac{\overline{\partial}}{2}:= \frac{1}{2} \left( \partial_{x_0}- \sum_{\ell=1}^{n} e_{\ell} \partial_{x_{\ell}} \right).
\end{equation}
\end{defn}
The domains of $ \mathcal{M}_{P_1^n}$ and $\frac{\overline{\partial}}{2}$ in $ \mathcal{F}( \mathbb{R}^{n+1})$ are defined as
$$ D \left( \mathcal{M}_{P_1^n}\right):= \left \{ f  \in \mathcal{F}( \mathbb{R}^{n+1}), \quad \mathcal{M}_{P_1^n}(f) \in \mathcal{F}( \mathbb{R}^{n+1}) \right\},$$

$$ D \left( \frac{\overline{\partial}}{2} \right):= \left \{ f  \in \mathcal{F}( \mathbb{R}^{n+1}), \quad \frac{\overline{\partial}}{2}(f) \in \mathcal{F}( \mathbb{R}^{n+1}) \right\}.$$

These are proper subsets of the Clifford-Appell-Fock space. Indeed let us consider
$$g(x):=\sum_{k=0}^{\infty} \frac{P_k^n(x)}{\sqrt{(k+1)(k+1)!}}.$$
Conversely, we have
$$ \mathcal{M}_{P_1^n}(g)(x)=\sum_{k=1}^{\infty} \frac{P_k^n(x)}{\sqrt{k k!}} \qquad \frac{\overline{\partial}}{2}(g)(x)=\sum_{k=0}^{\infty} \frac{(k+1)P_k^n(x)}{\sqrt{(k+2)(k+2)!}}.
$$
The function $g(x)$ belongs to the Clifford-Appell-Fock space. Indeed 
$$\sum_{k=0}^{\infty}\frac{k!}{(k+1)!(k+1)}=\sum_{k=0}^\infty\frac{1}{(k+1)^2}<\infty.$$
However $\mathcal{M}_{P_1^n}(g)(x)$ and $\frac{\overline{\partial}}{2}(g)(x)$ do not belong to the Fock space, indeed 
$$ \sum_{k=1}^{\infty}k! \frac{1}{k k!}=\sum_{k=1}^\infty\frac{1}{k}= +\infty,$$
and
$$ \sum_{k=0}^\infty k! \frac{(k+1)^2}{(k+2)(k+2)!}= \sum_{k=0}^\infty \frac{(k+1)}{(k+2)^2}= + \infty.$$

\begin{prop}
\label{510}	
Let $f\in D(\frac{\overline{\partial}}{2})$ and $g\in D(\mathcal{M}_{P_1^n})$. Then, we have
$$\scal{\frac{\overline{\partial}}{2}(f),g}_{\mathcal{F}(\mathbb{R}^{n+1})}=\scal{f,\mathcal{M}_{P_1^n}g}_{\mathcal{F}(\mathbb{R}^{n+1})}.$$
In particular, the adjoint operator of the hypercomplex derivative is given by $\mathcal{M}_{P_1^n}$, i.e. we have $$\displaystyle (\mathcal{M}_{P_1^n})^*=\frac{\overline{\partial}}{2}.$$
\end{prop}
\begin{proof}
Let us write $f$, $g \in \mathcal{F}(\mathbb{R}^{n+1})$ as
$$ f(x)= \sum_{k=0}^\infty P_k^n(x) \alpha_k, \qquad \{\alpha_k\}_{k \in \mathbb{N}_0} \subset \mathbb{R}_n,$$
$$ g(x)= \sum_{k=0}^\infty P_k^n(x) \beta_k, \qquad \{\beta_k\}_{k \in \mathbb{N}_0} \subset \mathbb{R}_n.$$
Applying formula \eqref{app} we get
\begin{eqnarray*}
\frac{\overline{\partial}}{2}(f(x)) &=& \sum_{k=1}^\infty kP_{k-1}^n(x) \alpha_k\\
&=& \sum_{h=0}^\infty P_{h}^n(x) (h+1)\alpha_{h+1}.
\end{eqnarray*}
Therefore
\begin{equation}
\label{one}
\left \langle \frac{\overline{\partial}}{2}(f),g \right \rangle_{\mathcal{F}(\mathbb{R}^{n+1})}= \sum_{k=0}^\infty (k+1)! \overline{\alpha_{k+1}} \beta_k.
\end{equation}
On the other side, by changing indexes in formula \eqref{shift}, we get
\begin{equation}
\label{star1}
\left(\mathcal{M}_{P_1^n} g\right)(x)=\sum_{h=1}^\infty P^n_{h}(x) \beta_{h-1}.
\end{equation}

Thus
\begin{equation}
\label{two}
\left \langle f, \mathcal{M}_{P_1^n} \right \rangle_{\mathcal{F}(\mathbb{R}^{n+1})}= \sum_{\ell=1}^\infty \ell! \overline{\alpha_{\ell}} \beta_{\ell-1}.
\end{equation}
By changing indexes with $ \ell= k+1$, and comparing with \eqref{one}, we prove the statement.
\end{proof}
\begin{prop}
We have the following commutation rule between the annihilation and creation operator.
$$ \left [\mathcal{M}_{P_1^n},\frac{\overline{\partial}}{2}\right]=\mathcal{I},$$
where $[.,.]$ is the commutator operator.
\end{prop}
\begin{proof}
Given a function $ f(x)= \sum_{k=0}^\infty P_k^n(x) \alpha_k$, with $\{\alpha_k\}_{k \in \mathbb{N}_0} \subset \mathbb{R}_n$, such that $ \mathcal{M}_{P_1^n}f \in D \left( \frac{\bar{\partial}}{2}\right)$ and $\frac{\bar{\partial}}{2}f \in D \left(\mathcal{M}_{P_1^n}\right)$, we have to show that
$$ \frac{\overline{\partial}}{2}\mathcal{M}_{P_1^n}f-\mathcal{M}_{P_1^n}\frac{\overline{\partial}}{2}f=f.$$
By formula \eqref{app}, and using similar arguments of the proof of Proposition \ref{510}, we get 
$$ \frac{\overline{\partial}}{2} f(x)= \sum_{k=0}^\infty (k+1) P_k^n(x )\alpha_{k+1}.$$
By Lemma \ref{prod} we have
\begin{eqnarray}
\label{three}
\mathcal{M}_{P_1^n}\frac{\overline{\partial}}{2} f(x)&=& \sum_{k=0}^\infty (k+1) P_{k+1}^n(x) \alpha_{k+1}\\
\nonumber
&=&\sum_{k=1}^\infty k P_{k}^n(x) \alpha_{k}.
\end{eqnarray}
Furthermore
\begin{eqnarray}
\label{four}
\frac{\overline{\partial}}{2} \mathcal{M}_{P_1^n}f(x) &=&\frac{\overline{\partial}}{2} \left( \sum_{k=0}^\infty P^n_{k+1}(x) \alpha_{k}\right)\\
\nonumber
&=& \sum_{k=0}^\infty (k+1) P^n_k(x) \alpha_k\\
\nonumber
&=& \sum_{k=0}^\infty k P^n_k(x) \alpha_k+\sum_{k=0}^\infty P^n_k(x) \alpha_k.
\end{eqnarray}
Finally subtracting \eqref{three} with \eqref{four} we get
$$
\frac{\overline{\partial}}{2}\mathcal{M}_{P_1^n}f(x)-\mathcal{M}_{P_1^n}\frac{\overline{\partial}}{2}f(x)= \sum_{k=0}^\infty P_k^n(x) \alpha_k=f(x).
$$
\end{proof}

\section{The Clifford-Appell-Hardy space}
In this section, we extend the Hardy space to the monogenic setting. In the complex setting a holomorphic function in the unit disk $f(z)= \sum_{n=0}^\infty z^n a_n$ belongs to the Hardy space if $ \sum_{n=0}^\infty |a_{n}|^2 < \infty$. In \cite{ACS, ACSbook}, an extension of the Hardy space to the slice monogenic setting has been provided. In this section, our goal is to extend the Hardy space to the monogenic setting. As done for the Fock space, we will continue to use the generalized CK-extension for the same reasons given in Section 5.
\\Now, let us consider the Euclidean unit ball $ \mathbb{B}$ on $ \mathbb{R}^{n+1}$ and a fixed odd number $n$. We define the Clifford-Appell-Hardy space as
\begin{equation}
\label{space11}
\mathcal{H}_n(\mathbb{B}):=\left\lbrace{ \sum_{k=0}^{\infty}P_k^n(x)\alpha_k, \quad\alpha_k\in\mathbb{R}_n,\quad\sum_{k=0}^{\infty}|\alpha_k|^2<\infty}\right\rbrace.
\end{equation}
From the inequality $|P^n_k(x)| \leq |x|^k$ it follows that the Clifford-Appell Hardy is defined in the Euclidean unit ball, like for the complex Hardy space.
\begin{rem}
Since for any $k \geq 0$ and fixed $n$ odd we have that $ \frac{[k!]^2}{[(k+n-1)!]^2} \leq 1$, we have that the space introduced in \eqref{space11} is contained in the one defined in \eqref{space1}.
\end{rem}
As usual we define the Hilbert module structure as follows.

\begin{defn}
Let $n$ be a fixed odd number. Let us consider $f(x)=\sum_{k=0}^{\infty}P_k^n(x)\alpha_k$ and $g(x)=\sum_{k=0}^{\infty}P_k^n(x)\beta_k$ in $ \mathcal{H}_n(\mathbb{B})$. We define the scalar product as
$$ \langle f,g \rangle_{\mathcal{H}_n(\mathbb{B})}= \sum_{k=0}^{\infty} \overline{\alpha}_k \beta_k.$$
\end{defn}
Similarly, we obtain the following consequence of Theorem \ref{Chap}, with $b_k:=1$ and $ \Omega_b= \mathbb{B}$.
\begin{thm}
Let $n$ be a fixed odd number. Then for any $f \in \mathcal{H}_{n}(\mathbb{B})$ we have
$$
|f(x)| \leq \left( \sum_{k=0}^\infty |x|^{2k} \right)^{\frac{1}{2}} \| f\|_{\mathcal{H}_{n}(\mathbb{B})}\\
=\sqrt{\frac{1}{1-|x|^2}}\| f\|_{\mathcal{H}_{n}(\mathbb{B})}.
$$
Moreover, the reproducing kernel of the Clifford-Appell-Hardy space $\mathcal{H}_n(\mathbb{B})$ is given by the convergent series
\begin{equation}
	\label{rep}
	K_H(x,y)=\sum_{k=0}^{\infty}\overline{P_k^n(x)}P_k^n(y), \quad x,y\in\mathbb{B}.
\end{equation}
\end{thm}
\begin{proof}
It follows by considering $b_k=1$ and $\Omega_b= \mathbb{B}$ in Theorem \ref{Chap}.
\end{proof}
\begin{rem}
As to this moment, we are not able to provide a closed explicit formula for the sum \eqref{rep}. However, if we consider $x_0$, $y_0 \in (-1,1)$ we have
$$K_H(x_0,y_0)= \frac{1}{1-x_0y_0},$$
which is the reproducing kernel of the classic Hardy space.
\end{rem}
Now, we extend the notions of shift and backward shift operators in the Clifford algebras setting. Inspired from Definition \ref{cre} we have the following.
\begin{defn}
Let $n$ be a fixed odd number. The shift operator acting on a function $f \in \mathcal{H}_n(\mathbb{B})$ is defined as follows
\begin{equation}
\mathcal{M}_{P_1^n}(f):= P_1^n \odot_{GCK} f .
\end{equation}
\end{defn}
\begin{defn}
Let $n$ be a fixed odd number. For $f\in \mathcal{H}_n(\mathbb{B})$ we define backward shift operator as
\begin{equation}
S(f):=\left(P_1^n\right)^{-\odot_{GCK}}\odot_{GCK} f.
\end{equation}
\end{defn}

\begin{prop}
\label{GCK3}
For every $k \in \mathbb{N}$, we have
\begin{equation}
\left(P_1^n(x)\right)^{-\odot_{GCK}}\odot_{GCK}P_k^n(x)=P^{n}_{k-1}(x), \quad \forall x\in \mathbb{B}.
\end{equation}
\end{prop}
\begin{proof}
By recalling the definition of GCK product, see Definition \ref{np}, Definition \ref{inve} and by formula \eqref{GCK2} we get
\begin{align*}
\left(P_1^n(x)\right)^{-\odot_{GCK}}\odot_{GCK}P_k^n(x) &= GCK\left(\frac{1}{x_0}\right)\odot_{GCK} GCK(x_0^k) \\
&=GCK(x_{0}^{k-1}) \\
&=P_{k-1}^{n}(x).
\end{align*}
\end{proof}

From previous result is clear that the operator $S$  is the backward shift operator with respect to the Clifford-Appell system $ \{P_k^n\}_{k \geq 0}$. Indeed, let $f(x)= \sum_{k=0}^{\infty}P_k^n(x) \alpha_k$ in $\mathcal{H}_n(\mathbb{B})$. 
By Proposition \ref{GCK3} we can write 

\begin{equation}
\label{GCK4}
S(f)(x)=\sum_{k=1}^{\infty}P_{k-1}^{n}(x)\alpha_k.
\end{equation}

In the next result we show that the backward shift operator is the adjoint of the shift operator.
\begin{thm}
Let us assume $n$ be a fixed odd number and $f$, $g \in \mathcal{H}_n(\mathbb{B})$. Then it holds that
\begin{equation}
\langle S(f),g \rangle_{\mathcal{H}_n(\mathbb{B})}=\langle f, \mathcal{M}_{P_1^n}(g) \rangle_{\mathcal{H}_n(\mathbb{B})}.
\end{equation}
This means that
$$ (\mathcal{M}_{P_1^n})^*= S.$$
\end{thm}
\begin{proof}
Let $f(x)= \sum_{k=0}^{\infty}P_k^n(x) \alpha_k$ and $g(x)= \sum_{k=0}^{\infty}P_k^n(x)\beta_k$. By changing indexes in formula \eqref{GCK4} we get
\begin{equation}
\label{star4}
S(f)(x)= \sum_{k=0}^{\infty}P_k^n(x) \alpha_{k+1}.
\end{equation}
Moreover, by formula \eqref{star1} we have
$$(\mathcal{M}_{P_1^n}(g))(x)= \sum_{k=1}^{\infty}P_k^n(x) \beta_{k-1}.$$
Finally, after some manipulations we obtain
 \begin{align*}
\langle S(f),g\rangle_{\mathcal{H}_n(\mathbb{B})} &= \sum_{k=0}^{\infty} \overline{\alpha_{k+1}}\beta_k \\
&= \sum_{k=1}^{\infty} \overline{\alpha_{k}}\beta_{k-1}\\
&=  \langle f, \mathcal{M}_{P_1^n}(g)\rangle_{\mathcal{H}_n(\mathbb{B})}.
\end{align*}
\end{proof}


\section{Clifford-Appell polyanalytic polynomials in the Clifford case}
In this last section we want to address the following problem.
\newline
\newline
\textbf{Problem} Is it possible to extend to higher order the Clifford-Appell polynomials to the setting of polyanalytic monogenic functions?
\newline
\newline
In order to answer to the previous question we need to introduce the theory of polyanalytic functions. This is one of the various ways to extend the theory of holomorphic functions to higher order, see \cite{B}. Recently, the notion of slice monogenic function has been extended in the polyanalytic setting, see \cite{ACDS2, ADS2019}. Before to go through this section we need to recall some notions.
\begin{defn}[Poly slice function]
	Let $m \in \mathbb{N}_0$ and let $U$ be a an axially slice symmetric open set in $ \mathbb{R}^{n+1}$. A function $f:U \to \mathbb{R}_n$ in $C^{m+1}(U)$ is called left slice polyanalytic function of order $m+1$, if it is a (left) slice function and it satisfies the following equation
	$$ \frac{1}{2^{m+1}}(\partial_u+ I \partial_v)^{m+1} f(u+Iv)=0, \qquad \hbox{for all} \quad I \in \mathbb{S}^{n-1}.$$
\end{defn}
The left module of slice polyanalytic functions of order $m+1$ will be denoted by $ \mathcal{SP}_{m+1}(U)$.

We observe that if we put $m=0$ in the previous definition we get the slice monogenic functions, see Definition \ref{moons}. Many properties of slice polyanalytic functions of order $m$ are investigated in \cite{ADS2019}. We recall one of the most important.
\begin{prop}[polyanalytic decomposition]
	\label{deco}
	Let $ \Omega$ be an axially symmetric slice domain of $ \mathbb{R}^{n+1}$. A function $f: \Omega \to \mathbb{R}_n$ is slice polyanalytic of order $m+1$ if and only if there exists unique slice monogenic functions $f_0$,..., $f_m$ on $ \Omega$ such that we have the following decomposition
	$$ f(x)= \sum_{k=0}^m \bar{x}^{k} f_{k}(x); \qquad \forall x \in \Omega.$$
\end{prop}
Now, we recall the notion of polyanalytic monogenic function, see \cite{B1976, DB1978}.
\begin{defn}
	Let $ U \subset \mathbb{R}^{n+1}$ be an open set and let $f:U \to \mathbb{R}_n$ be a function of class $ C^{m+1}(U)$. We say that $f$ is left polyanalytic monogenic of order $m+1$ on $U$ if
	$$ \partial^{m+1}f(x)= \left( \partial_{x_0}+ \sum_{j=1}^n e_j \partial_{x_j}\right)^{m+1}f(x)=0, \qquad \forall x \in U.$$
\end{defn}
The left module of axially polyanalytic monogenic functions of order $m+1$ will be denoted by $\mathcal{AM}_{m+1}(U)$. In the polyanalytic monogenic setting holds the following characterization, see \cite{B1976, DMD6}.
\begin{prop}
\label{polydeco}
Let $ U \subset \mathbb{R}^{n+1}$ be an open. A function $f:U \to \mathbb{R}_n$ is polyanalytic of order $m$ if and only if can be decomposed in terms of some unique monogenic functions $f_0,\ldots, f_{m-1}$
$$ f(x)= \sum_{k=0}^{m-1} x_0^k f_k(x).$$
Moreover, the monogenic projections can be computed as
\begin{equation}
f_k(x)= \frac{1}{k!} \sum_{s=0}^{m-k} \frac{(-x_0)^{m-k-s}}{(m-k-s)!} \partial^{m-s}(f)(x).
\end{equation}
\end{prop}
Also in the polyanalytic case it is possible to relate the two kind of polyanalytic regularities by a means of Fueter-Sce methodology, see \cite{ADS}. This has been achieved in two ways. Now, we introduce the first polyanalytic Fueter-Sce map
\begin{defn}
	Let $n$ be an odd number. Let $\Omega$ be an axially symmetric slice domain and $f\in\mathcal{SP}_{m+1}(\Omega)$ with a polyanalytic decomposition given by $f(x)= \sum_{k=0}^{m}\overline{x}^kf_k(x)$, where $f_k\in\mathcal{SM}(\Omega)$. Then, the first polyanalytic Fueter-Sce map $\mathcal{C}_{m+1}:\mathcal{SP}_{m+1}(\Omega)\longrightarrow \mathcal{AM}_{m+1}(\Omega)$ is defined by $$\mathcal{C}_{m+1}(f)(x):=\sum_{k=0}^mx_0^k\Delta^{\frac{n-1}{2}}_{\mathbb{R}^{n+1}} f_k(x), \quad \forall x\in\Omega.$$
\end{defn}

From the Fueter-Sce Theorem it follows that all the functions in the sequence $ \{\Delta^{\frac{n-1}{2}}_{\mathbb{R}^{n+1}} f_k(x) \}_{0 \leq k \leq m}$ are axially monogenic. Therefore by Proposition \ref{polydeco} the map $ \mathcal{C}_{m+1}$ is well-defined. From the same result we can write 
\begin{equation}
\label{line}
\mathcal{C}_{m+1}(f)(x):=\sum_{k=0}^m \sum_{s=0}^{m-k} \frac{(-1)^k}{k!}  \frac{(-x_0)^{m-s}}{(m-k-s)!} \partial^{m-s} \Delta_{\mathbb{R}^{n+1}}^{\frac{n-1}{2}}(f)(x),
\end{equation}
which obviously implies that $ \mathcal{C}_{m+1}$ is $ \mathbb{R}_n$- (right) linear.
\\In the polyanalytic setting there is another Fueter-Sce map, denoted by $\tau_{m+1}$, which has a different range from $ \mathcal{C}_{m+1}$. 
\begin{defn}
Let $ \Omega \subseteq \mathbb{R}^{n+1}$ be an axially symmetric slice domain, then we define the second polyanalytic Fueter-Sce map $\tau_{m+1} : \mathcal{SP}_{m+1}(\Omega) \longrightarrow \mathcal{AM}(\Omega)$ as
$$ \tau_{m+1}(f)=\Delta_{\mathbb{R}^{n+1}}^{\frac{n-1}{2}} V^m (f)(q).$$
where $$V(f)(x)= \frac{\partial}{\partial x_0}f(x)+ \frac{\underline{x}}{| \underline{x}|^2} \sum_{\ell=1}^n x_{\ell} \frac{\partial}{\partial x_\ell} f(x) \qquad x \in \Omega \setminus \mathbb{R}$$
is the so-called global operator with nonconstant coefficients.
\end{defn}
From the definition of the map $ \tau_{m+1}$ it is clear that it is $ \mathbb{R}_n$-(right) linear
\begin{rem}
Functions in the kernel of the operator $V$ are slice monogenic functions, see \cite{CGS2012, CGS2013, GP}. The main difference with Definition \ref{slicemono1} is that the operator $V$ is slice independent, for this reason this operator is known as global.
\end{rem}
The two polyanalytic Futer-Sce maps are connected by the following result, see \cite[Thm 3.13]{ADS}.
\begin{thm}
\label{rela}
Let $n$ be a fixed number. Let $f: \Omega \longrightarrow \mathbb{R}_n$ be a slice polyanalytic function of order $m + 1$ on the axially symmetric slice domain $ \Omega$. Then we have
\begin{equation}
\partial^m \mathcal{C}_{m+1}(f)(x)= \frac{1}{2^m} \tau_{m+1}(f)(x),\quad \forall x \in \Omega.
\end{equation}
\end{thm}
We can summarize the action of the two polyanalytic Fueter-Sce maps by the following diagram

\[
\begin{tikzcd}
	& \mathcal{AM}_{m+1}(\Omega) \arrow{dr}{2^m \partial^m} \\
	\mathcal{SP}_{m+1}(\Omega) \arrow{ur}{\mathcal{C}_{m+1}} \arrow{rr}{\tau_{m+1}} && \mathcal{AM}(\Omega)
\end{tikzcd}
\]
\newline
\newline
The aim of this section is to show what happens when we apply the polyanalytic Futer-Sce maps to a slice polyanalytic function of order $m+1$ developed in series. We will see that this action leads to a new family of polynomials that we call Clifford-Appell polyanalytic polynomials. 
\\Let $\Omega$ be a slice domain such that $0 \in \Omega$ and $f\in\mathcal{SP}_{m+1}(\Omega)$ be a slice polyanalytic function of order $m+1$ on $\Omega$. The polyanalytic decomposition (see Theorem \ref{deco}) states that the function $f$ can be written uniquely as
$$f(x)= \sum_{k=0}^{m}\overline{x}^kf_k(x), \quad x\in \Omega,$$
where $f_k$ are slice monogenic functions. Using the series expansion theorem for slice monogenic functions (see Theorem \ref{expa}) we can write each $f_k$ in a suitable neighbourhood of zero as $ f_k(x)=\sum_{j=0}^{\infty}x^j\alpha_{k,j}$ with $\lbrace{\alpha_{k,j}}\rbrace_{0 \leq k\leq m, j\in\mathbb{N}_0}\subset\mathbb{R}_n$. Thus, we obtain the following series expansion
\begin{equation}\label{series}
f(x)=\sum_{k=0}^{m}\sum_{j=0}^{\infty}\overline{x}^kx^j \alpha_{k,j}.
\end{equation}
\underline{\emph{The action of the map $\mathcal{C}_{m+1}$}}
\newline
\newline
Now, we apply the map $\mathcal{C}_{m+1}$ to the expansion \eqref{series} and by the linearity of this map we get
\begin{equation}
\mathcal{C}_{m+1}(f):=\sum_{k=0}^{m}\sum_{j=0}^{\infty}\mathcal{C}_{m+1}(\overline{x}^kx^j)\alpha_{k,j}.
\end{equation}
To fully describe the action of $ \mathcal{C}_{m+1}$, it is enough to describe how it acts on the monomials of the form $\overline{x}^kx^j$ with $k=0,...,m$ and $j\geq 0$.
\begin{thm}\label{C-n action}
\label{poly}
Let $n$ be a fixed odd number and $m\geq 0$. Then, for any $0 \leq k \leq m$ we have
$$
\mathcal{C}_{n+1}(\bar{x}^k x^j)=
\begin{cases}
0 & \mbox{if } \quad j < n-1, \\
x_0^k 4^{\frac{n-1}{2}}(-1)^{\frac{n-1}{2}} \left[ \Gamma \left( \frac{n+1}{2}\right)\right]^{2}& \mbox{if } \quad j=n-1,\\
\left[ \Gamma \left( \frac{n+1}{2}\right)\right]^{2} 2^{n-1}(-1)^{\frac{n-1}{2}} \frac{(n)_{j+1-n}}{(j+1-n)!} x_0^k P^n_{j+1-n}(x), & \mbox{if} \quad j \geq n-1,
\end{cases}
$$
where the polynomials $P^n_{j+1-n}(x)$ are defined in formula \eqref{appe}.
\end{thm}

\begin{proof}
From the definition of the map $ \mathcal{C}_{m+1}$ it follows that	
$$ \mathcal{C}_{m+1}(\overline{x}^{k}x^j)=x_0^k\Delta_{\mathbb{R}^{n+1}}^{\frac{n-1}{2}}(x^j).$$
Now, the result follows by Corollary \ref{lapal2}.
\end{proof}

\begin{rem}
This is an extension to general Clifford algebras of the result \cite[Thm. 3.3]{DMD}, established in the quaternionic setting. Indeed, if we consider $n=3$ in Theorem \ref{poly}, for $q \in \mathbb{H}$ and $j \geq 2$, we get
\begin{eqnarray*}
\mathcal{C}_{m+1}(\overline{q}^{k}q^j)&=& -[\Gamma(2)]^2 2^2 \frac{(3)_{j-2}}{(j-2)!}x_0^k P^3_{j-2}(q)\\
&=& -4 \frac{\Gamma(j+1)}{\Gamma(3) (j-2)!} x_0^k P^3_{j-2}(q)\\
&=& -2 j(j-1) x_0^k P^3_{j-2}(q).
\end{eqnarray*}
which is exactly the same result obtained in \cite{DMD}.
\end{rem}
The expansion given in Theorem \ref{poly} yields a set of generators for the space of polyanalytic functions of axial type. These generators are given by
\begin{equation}
\label{gener}
\mathcal{A}_{k,s}^n(x):=x_0^kP^n_{s}(x),\quad k=0,...,m, \quad s\geq 0.
\end{equation}
We shall show that this family of polynomials also satisfy an Appell-like property.
\begin{prop}
	Let $ 0 \leq k \leq m$ and $ s \geq 0$. Then we have
\begin{equation}
\label{app4}
\overline{\partial} [\mathcal{A}_{k,s}^n(x)]=k \mathcal{A}_{k-1,s}^n+2s \mathcal{A}_{k,s-1}^n.
\end{equation}
\end{prop}
\begin{proof}
This result follows directly from the Leibniz formula and the Appell property for the polynomials $P_k^n(x)$, see e.g., formula \eqref{app}. Indeed,
\begin{eqnarray*}
\overline{\partial} [\mathcal{A}_{k,s}^n(x)]&=& \overline{\partial}[x_0^kP^n_{s}(x)]\\
&=& k x_{0}^{k-1}P^n_{s}(x)+ 2s x_{0}^k P^n_{s-1}(x)\\
&=& k \mathcal{A}_{k-1,s}^n+2s \mathcal{A}_{k,s-1}^n.
\end{eqnarray*}
\end{proof}
The result proved in the above Proposition can be considered an Appell-like property for the generalized Clifford-Appell polyanalytic polynomials. Indeed, the polynomials $\mathcal{A}^n_{k,s}(x)$ are such that $ \mathcal{A}_{0,s}^n(x)=P^n_s(x)$. Thus, by setting $m=0$ (this implies $k=0$) in formula \eqref{app4} we get the Appell property \eqref{app}.
\newline
\newline
We finish the discussion on the application of the map $ \mathcal{C}_{m+1}$ by writing a polyanalytic function of order $m+1$ in terms of the generalized Clifford-Appell polyanalytic polynomials $ \mathcal{A}_{k,s}^n(x)$.
\begin{thm}
Let us consider $n$ a fixed odd number. Let $\Omega$ be an axially symmetric slice domain.
The image of $g(x)= \sum_{s=0}^\infty x^s \alpha_s$, with $ \{\alpha_s\}_{s \geq 0} \subset \mathbb{R}_n$, under the map $ \mathcal{C}_{m+1}$ in $ \mathcal{SP}_{n+1}(\Omega)$ is given by
\begin{equation}
f(x)=\mathcal{C}_{m+1}(g)(x)= \sum_{k=0}^{m}\sum_{s=0}^{\infty}\mathcal{A}^n_{k,s}(x)\beta_{k,s},
\end{equation}	
where $ \beta_{k,s}=\left[ \Gamma \left( \frac{n+1}{2}\right)\right]^{2} 2^{n-1}(-1)^{\frac{n-1}{2}}\frac{(n)_{s}}{s!} \alpha_{k,s+n-1}$.
\end{thm}
\begin{proof}
It follows by putting together formula \eqref{series}, Theorem \ref{poly} and formula \eqref{gener}.
\end{proof}
\underline{\emph{The action of the map $\tau_{m+1}$}}
\newline
\newline
Now, we apply the second polyanalytic Fueter-Sce map to the series expansion \eqref{series}. From the linearity of this map it follows
\begin{equation}
\label{series2}
\tau_{m+1}(f)= \sum_{k=0}^m \sum_{j=0}^\infty \tau_{m+1}(\bar{x}^k x^j) \alpha_{k,j}.
\end{equation}
Now, we show how the map $ \tau_{m+1}$ acts on the monomials $\bar{x}^k x^j$.
\begin{thm}
\label{BL1}
Let $n$ be a fixed odd number. For any fixed $m\geq 0$ and $0 \leq k \leq m$ we have
\begin{equation}
\tau_{m+1}(\bar{x}^k x^j)=
\begin{cases}
0 & \mbox{if } \quad j< n-1, \\
2^{m+n-1} m! (-1)^{\frac{n-1}{2}} \left[\left(\frac{n-1}{2}\right)!\right]^2\delta_{m,k}  & \mbox{if } \quad j= n-1, \\
2^{m+n-1} m! \left[\Gamma \left( \frac{n+1}{2}\right)\right]^2 \frac{(n)_{j-n+1}}{(j-n+1)!} (-1)^\frac{n-1}{2} P^{n}_{j-n+1}(x) \delta_{m,k} & \mbox{if} \quad j > n-1,
\end{cases}
\end{equation}
where the polynomials $P^n_{j+1-n}(x)$ are defined in formula \eqref{appe}.
\end{thm}
\begin{proof}
First of all, we recall that a slice monogenic polyanalytic function $f$ can be written using a decomposition in terms of slice monogenic functions $ \{f_k\}_{0 \leq k \leq m}$
\begin{equation}
\label{summ}
f(x)= \sum_{k=0}^m \bar{x}^k f_k(x),
\end{equation}
see Theorem \ref{polydeco}. By \cite[Proposition 3.4]{ADS} (one can easily extend in $ \mathbb{R}^{n+1}$) we know the following property of the global operator
$$ V(\bar{x}^kf)(x)=2k\bar{x}^{k-1}f(x).$$
Then we have
$$ V(f)(x)=\sum_{k=1}^{m} 2k \bar{x}^{k-1} f_k(x).$$
Similarly, we apply $m$-times the golbal operator $V$ to obtain
$$ V^m f(x)=2^m m! f_{m}(x),$$
where $f_m(x)$ is the last term of the sum defined in \eqref{summ}.
\\ This implies that
$$ V^{m}\left(\bar{x}^k x^j \right)=2^m m! x^j \delta_{m,k}.$$
From the definition of the map $ \tau_{m+1}$ we get
\begin{eqnarray*}
\tau_{m+1}(\bar{x}^k x^j) &=& \Delta_{\mathbb{R}^{n+1}}^{\frac{n-1}{2}} V^m \left(\bar{x}^k x^j\right) \\
&=& 2^m m!  \delta_{m,k} \Delta_{\mathbb{R}^{n+1}}^{\frac{n-1}{2}} x^j.
\end{eqnarray*}
Now, the result follows by Corollary \ref{lapal2}.
\end{proof}
\begin{rem}
If we put $n=3$ in Theorem \ref{BL1}, we get the same result obtained in \cite[Thm. 3.14]{DMD}.
\end{rem}

\begin{prop}
	\label{res1}
	Let $n$ be a fixed odd number. Then for any fixed $m\geq 0$ and $j > n-1$ we have
	$$\partial^m \mathcal{A}^n_{k,j+1-n}=m! \delta_{m,k} P^{n}_{j-n+1}(x).$$
\end{prop}
\begin{proof}
By Theorem \ref{poly} for $j>n-1$ we get
$$ \mathcal{A}^n_{k,j+1-n}(x)=\frac{(-1)^{\frac{n-1}{2}}(j+1-n)!}{\left[\Gamma \left(\frac{n+1}{2}\right) \right]^2 2^{n-1}(n)_{j+1-n} } \mathcal{C}_{m+1}(\bar{x}^k x^j).$$
Finally, by Theorem \ref{rela} and Theorem \ref{BL1} we obtain
\begin{eqnarray*}
\partial^m(\mathcal{A}^n_{k,j+1-n}(x))&=&\frac{(-1)^{\frac{n-1}{2}}(j+1-n)!}{\left[\Gamma \left(\frac{n+1}{2}\right) \right]^2 2^{n-1}(n)_{j+1-n} } \partial^m\mathcal{C}_{m+1}(\bar{x}^k x^j)\\
&=& \frac{(-1)^{\frac{n-1}{2}}(j+1-n)!}{\left[\Gamma \left(\frac{n+1}{2}\right) \right]^2 2^{m+n-1}(n)_{j+1-n} } \mathcal{\tau}_{m+1}(\bar{x}^k x^j)\\
&=& m! \delta_{m,k} P^{n}_{j-n+1}(x).
\end{eqnarray*}
\end{proof}

In the following result we get a full description of the action of the map $ \tau_{m+1}$ to a slice monogenic polyanalytic function of order $m+1$.
\begin{thm}
\label{R5}
Let $n$ be a fixed odd number. Let us consider $\Omega$ an axially symmetric slice domain.
The image of $g(x)= \sum_{s=0}^\infty x^s \alpha_s$ with $ \{\alpha_s\}_{s \geq 0} \subset \mathbb{R}_n$, under the map $\tau_{m+1}$ in $ \mathcal{SP}_{n+1}(\Omega)$ is given by
$$ f(x)= \tau_{m+1}(g)(x)= \sum_{s=0}^\infty P^{n}_{s}(x) \gamma_{k,s+n-1},$$
where $ \gamma_{k,s}=2^{m+n-1}(-1)^\frac{n-1}{2} m! \left[\Gamma \left( \frac{n+1}{2}\right)\right]^2 \frac{(n)_{s}}{s!} \alpha_{k,s+n-1}$ .
\end{thm}
\begin{proof}
It follows by putting together formula \eqref{series2} and Theorem \ref{BL1}.
\end{proof}

\begin{rem}
In \cite{DMD6} it is showed that the two maps $ \tau_{m+1}$ and $ \mathcal{C}_{m+1}$ are surjective.
\end{rem}

\section{Concluding remarks}

In this work, we used the relation between the well-known Sce-Fueter mapping and generalized Cauchy-Kovalevskaya theorems in order to study a family of Appell polynomials in Clifford analysis, which are related to some elementary functions and reproducing kernel Clifford modules. In the last section of the paper, a similar approach built upon a polyanalytic Fueter-Sce map result allowed to introduce a new family of Appell polyanalytic monogenic polynomials. In a forthcoming work, we will study reproducing kernel polyanalytic monogenic modules built upon these new Appell polyanalytic monogenic polynomials and related topics.

\section*{Competing interests declaration}
The authors declare none.

\vspace{5mm}

Schmid College of Science and Technology, Chapman University One University Drive Orange, California 92866, USA
\\Email address: \textbf{ademartino@chapman.edu}
\newline
\newline
Schmid College of Science and Technology, Chapman University One University Drive Orange, California 92866, USA
\\Email address: \textbf{diki@chapman.edu}
\newline
\newline
Clifford Research Group, Department of Electronics and Information Systems, Faculty of Engineering and Architecture, Ghent University, Krijgslaan 281, 9000 Gent, Belgium.
\\Email address: \textbf{ali.guzmanadan@ugent.be}

\end{document}